\documentclass[11pt]{article}    % Specifies the document style.
\usepackage{amsmath, amssymb, amsthm,pdfsync}
\usepackage{verbatim}
\usepackage{cite}
\usepackage{graphicx}
\usepackage{subfig, caption,pdfsync}
\usepackage[top=1in, bottom=1in, left=.85in, right=.85in]{geometry}

\newtheorem{thm}{Theorem}
\newtheorem{lem}[thm]{Lemma}
\newtheorem{cor}[thm]{Corollary}
\newtheorem{prop}[thm]{Proposition}
\newtheorem{defn}[thm]{Definition}
\newtheorem{rem}[thm]{Remark}

\newcommand{\R}{\mathbb{R}}

\newcommand{\Z}{\mathbb{Z}}
\newcommand{\tn}{\tilde{\nabla}}

\newcommand{\eps}{\epsilon}
\DeclareMathOperator{\diam}{diam}
\DeclareMathOperator{\id}{Id}

\numberwithin{equation}{section}
\numberwithin{thm}{section}

\title{Pulsating Fronts in a 2D Reactive Boussinesq System}% Declares the document's title. 
\author{Christopher Henderson}

\begin{document}           % End of preamble and beginning of text.

\maketitle                 % Produces the title.

\begin{abstract}
We consider a reactive Boussinesq system with no stress boundary conditions in a periodic domain which is unbounded in one direction.  Specifically, we couple the reaction-advection-diffusion equation for the temperature, $T$, and the linearized Navier-Stokes equation with the Boussinesq approximation for the fluid flow, $u$.  We show that this system admits smooth pulsating front solutions that propagate with a positive, fixed speed.
\end{abstract}

\section{Introduction}\label{Intro}

%textwidth in cm: \printinunitsof{cm}\prntlen{\textwidth}
%oddsidemargine : \printinunitsof{cm}\prntlen{\oddsidemargin}
%even:  \printinunitsof{cm}\prntlen{\evensidemargin}

In this paper, we establish the existence of pulsating traveling front solutions to the reactive Boussinesq system with the no stress boundary conditions.  The reactant, or temperature, $T$, and the fluid velocity, $u$, satisfy a system composed of coupling the reaction-advection-diffusion  equation for $T$ and the linearized Navier-Stokes equation for $u$ via the Boussinesq approximation as below
\begin{equation}\label{the_PDE_unrotated}
	\begin{split}
		T_t + u \cdot \nabla T &= \Delta T + f(x,z,T)\\
		u_t-\Delta u + \nabla p &= T\hat{z}\\
		\nabla \cdot u &= 0.
	\end{split}
\end{equation}
Here, the reaction term, $f$, is smooth and ignition type, and $\hat{z}$ is the vertical unit vector in $\R^2$.  That is,
there exists a positive number $\theta_0 \in (0,1)$ such that
\begin{equation}\label{ignition}
	\begin{split}
		f(x,z,r) &=0 \qquad \mbox{ for all } r \leq \theta_0\\
		f(x,z,r) &> 0 \qquad \mbox{ for all } r \in (\theta_0, 1)\\
		f(x,z,r) &\leq 0 \qquad \mbox{ for all } r \geq 1.
	\end{split}
\end{equation}
We consider \eqref{the_PDE_unrotated} in a smooth, periodic cylinder $\Omega\subset \R^2$.  Specifically, there is a unit vector $\hat{k}$ and positive real numbers $\ell$ and $\lambda$ such that 
\[
	\Omega + \ell\hat{k} = \Omega ~~\text{and}~~ \Omega \subset \{(x,z)\in\R^2: |(x,z)\cdot \hat{k}^\perp| \leq \lambda\}.
\]
In addition, we require that $f$ is $\ell$-periodic in the direction of $\hat{k}$.  In addition, we assume that there exists a positive constant $C$ and two constants $r_1, r_2 \in (\theta_0, 1)$ such that $f(x,z,r) \geq C$ for all $(x,z,r)\in \Omega\times (r_1,r_2)$.  For ease of notation, when no confusion will arise, we will refer to $f(x,z,T)$ simply as $f(T)$.  Notice that the maximum principle implies that $0 \leq T \leq 1$.

In order to simplify the notation, it is convenient to rotate this domain to make it horizontal.  In other words, we change variables so that $\Omega \subset [-\lambda, \lambda]\times \R$ and so that $\Omega$ is $\ell$-periodic in $x$.  Then \eqref{the_PDE_unrotated} becomes
\begin{equation}\label{the_PDE}
	\begin{split}
		T_t + u \cdot \nabla T &= \Delta T + f(x,z,T)\\
		u_t-\Delta u + \nabla p &= T\hat{e}\\
		\nabla \cdot u &= 0,
	\end{split}
\end{equation}
where $u$ is the fluid velocity measured relative to the new coordinate system.  Here $\hat{e}$ is simply $\hat{z}$ rotated by the angle between $\hat{k}$ and $\hat{x}$.  Notice that $f$ is now $\ell$-periodic in $x$.  We will consider this problem with the Neumann boundary conditions for $T$ and the no stress boundary conditions for $u$.  Namely,
\begin{equation}\label{boundary_conditions}
	\begin{split}
		\frac{\partial T}{\partial \eta} &= 0 \mbox{ on } \partial\Omega\\
		u\cdot \eta &= 0 \mbox{ on } \partial\Omega\\
		\omega  &= 0 \mbox{ on } \partial\Omega,
	\end{split}
\end{equation}
where $\omega = \partial_z u_1 - \partial_x u_2$ and $u = (u_1, u_2)$.

%\subsubsection*{Background: Fronts in Reaction-Diffusion Equations}

The study of front propagation in reaction-diffusion equations dates
back to the pioneering works of Kolmogorov, Petrovskii and
Piskunov~\cite{KPP} and Fisher~\cite{Fisher}.  More recently, a lot of
studies considered reaction-advection-diffusion equations in a
prescribed flow.  For an overview of many of these results, we point
the reader to~\cite{XinB} and references found within.  These results
were obtained under the assumption that the fluid flow, $u$, was
prescribed, and hence, is unaffected by the change in temperature.  On
the other hand, the Boussinesq approximation as in
\eqref{the_PDE_unrotated} accounts for the density difference by a
buoyancy force in the equation for the fluid flow.  This is the term
in the right hand side of the second line of \eqref{the_PDE_unrotated}
and \eqref{the_PDE}.  In~\cite{MX}, Malham and Xin studied the
regularity problem for this system, with the full Navier-Stokes
equation governing the advection. Over the next decade, traveling
waves were shown to exist in systems similar to that in
\eqref{the_PDE_unrotated} when $\Omega$ is a flat, non-vertical
cylinder.  First, Berestycki, Constantin, and Ryzhik showed existence
of traveling waves in two dimensions when the fluid is governed by the
full Navier-Stokes equation with no stress boundary
conditions~\cite{BCR}.  Later, this was extended to include the no
slip boundary conditions in two and three dimensional slanted
cylinders in~\cite{CLR,LM}.  Finally, traveling waves were also shown
to exist in $n$-dimensional cylinders with no slip boundary conditions
with Stokes' equation governing the fluid flow in~\cite{Lew}.  To our
knowledge, the case of more general periodic domains, unbounded in one
direction, has not been studied.

In the case of flat cylinders, one has to worry that the traveling
waves are not those constructed in the previous theory.  Namely, if
the traveling waves are planar and thus depend only on $x$, one can
show that $u \equiv 0$ and thus, that the solutions are the same as
those developed in the reaction-diffusion equation long ago.  It turns
out that the existence of planar fronts depends on the alignment of
gravity and the domain.  To be more specific, if $\hat{e}$ in
\eqref{the_PDE} is not horizontal, then non-planar fronts exist.  On
there other hand, if $\hat{e}$ is horizontal, then the existence of
non-planar fronts depends on the Rayleigh number~\cite{BKV}.  In this paper,
non-planar fronts will not be an issue.  When the boundary is not
flat, as in this paper, non-trivial planar fronts will not satisfy the
boundary conditions for $T$.

\subsubsection*{Pulsating fronts}

When the setting of the problem involves either an inhomogeneous
medium or a non-flat infinite cylinder, traveling waves can not exist.
However, a generalization of the traveling wave solution, the
pulsating front, was defined first in~\cite{Shig}.  J. Xin gave the
first rigorous proof of such a front a few years later~\cite{Xin}.
For further results on the existence of pulsating fronts in various
settings see~\cite{BH_Period}.  In that paper, the authors prove the
existence of pulsating fronts in many settings under the condition
that the advection is periodic and prescribed.  In this paper, we
define a pulsating front to be a solution to \eqref{the_PDE} such that
there is some positive constant $c>0$, called the front speed, such
that
\begin{equation}\label{pulsating_front}
	\begin{split}
		T(t + \frac{\ell}{c}, x, z) &= T(t, x - \ell, z)\\
		u(t + \frac{\ell}{c}, x, z) &= u(t, x - \ell, z),
	\end{split}
\end{equation}
and there is some constant $\theta_- \in (0,1]$ such that
\begin{equation*}
	\begin{split}
		\lim_{x\to\infty} T(t,x,z) &= 0\\
		\lim_{x\to-\infty} T(t,x,z) &= \theta_-\\
		\lim_{x\to\infty} u(t,x,z) &= 0\\
		\lim_{x\to-\infty} u(t,x,z) &= 0.
	\end{split}
\end{equation*}

\subsubsection*{The Moving frame}

Because of the conditions in \eqref{pulsating_front}, it is natural to look for functions in the moving frame.  That is, we wish to find $T^m(s,x,z)$ and a real number $c$ such that $T^m$ is periodic in $x$ and satisfies $T(t,x,z) = T^m(x - ct, x, z)$.  We make the same change of variables for the functions $u$, $\omega$, and $\Psi$ to get the function $u^m$, $\omega^m$, and $\Psi^m$, respectively.  We define the moving frame as the set
\[
	\R\times \Omega_p = \{(s,x,z) \in \R^3: (x,z) \in \Omega, 0 \leq x \leq \ell\}.
\]
The set $\Omega_p$ is the period cell of $\Omega$.  Let $P = \{(x,z)\in \Omega_p: x = 0,\ell\}$ be the periodic boundary of $\Omega_p$.  Let $B$ be the complement, $B = \partial \Omega_p \setminus P$.
\begin{figure*}
\begin{center}
	\includegraphics[scale=.3]{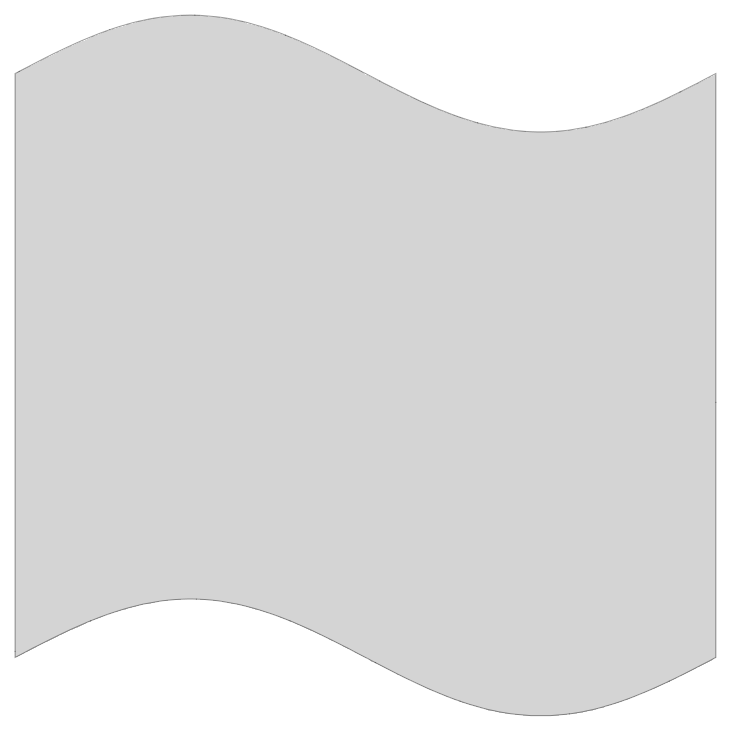}
	\caption*{The period cell $\Omega_p$}
\end{center}
\end{figure*}
This change of variables leads to the following system of equations
\begin{equation*}
	\begin{split}
		-cT^m_s + u^m \cdot \tn T^m &= L T^m + f(T^m)\\
		-cu^m_s - Lu^m + \tn p &= T^m \hat{e}\\
		\tn \cdot u^m &= 0,
	\end{split}
\end{equation*}
where
\begin{equation}\label{definition_of_operator}
	\begin{split}
		L &= \tn \cdot \tn,\\
		\tn &= (\partial_x + \partial_s, \partial_z).
	\end{split}
\end{equation}
The functions satisfy the following boundary conditions
\begin{equation*}
	\begin{split}
		\eta \cdot \tn T^m &= 0 \quad\mbox{ on } B\\
		\lim_{s\to\infty} T^m(s, x, z) &= 1 \quad\mbox{ uniformly on } \Omega_p\\
		\lim_{s\to\infty} T^m(s,x,z) &= 0 \quad\mbox{ uniformly on } \Omega_p\\
		u^m \cdot \eta &= 0 \quad\mbox{ on } B\\
		\omega^m &= 0 \quad\mbox{ on } B,
	\end{split}
\end{equation*}
and where all the function satisfy periodic boundary conditions on $P$.

We will use the following stream function formulation for $u$.  For each $t$, we let $u = \nabla^\perp \Psi = (\Psi_z, - \Psi_x)$ for $\Psi$ which solves the following system of equations:
\begin{equation}
	\begin{split}
		\Delta \Psi &= \omega \quad \mbox{ on } \Omega\\
		\Psi &= 0 \quad \mbox{ on } \partial\Omega\\
		\lim_{s\to \infty} \Psi(t,x,z) &= 0 \quad\mbox{ uniformly in } z\\
		\lim_{s\to -\infty} \Psi(t,x,z) &=0 \quad\mbox{ uniformly in } z
	\end{split}
\end{equation}
In the moving frame we represent $\Psi$ with $\Psi^m$, which solves the following equations:
\begin{equation}
	\begin{split}
		L \Psi^m &= \omega^m  \\
		\Psi^m &= 0 \quad\mbox{ on } \R\times B\\
		\lim_{s\to \infty} \Psi^m(s,x,z) &= 0 \quad\mbox{ uniformly on } \Omega_p\\
		\lim_{s\to -\infty} \Psi^m(s,x,z) &=0 \quad\mbox{ uniformly on } \Omega_p.
	\end{split}
\end{equation}
We impose periodic boundary conditions on $P$.

If we have a solution $(c,T^m,u^m)$ to the above system of equations, then $(c,\tilde{T}, \tilde{u})$ is also a solution, where $\tilde{T}^m(t,x,z) = T^m(s+s_0, x, z)$ and $\tilde{u}^m(s+s_0, x, z)$.  Hence, we also impose the extra condition that
\begin{equation}\label{normalization}
	\max_{s\geq 0, (x,z)\in \Omega_p} T^m(s,x,z) = \theta_0.
\end{equation}

\subsubsection*{The main result}

Our main result is the following theorem.

\begin{thm}\label{the_theorem}
Let the nonlinearity, $f$, be ignition type as in \eqref{ignition}.  Then there exists a pulsating front solution $(c,T, u)$ to the system \eqref{the_PDE} - \eqref{boundary_conditions}.  The solutions satisfy the following: $c>0$, $T \in C^{1+\alpha, 2+\alpha}$, $u\in C^{1+\alpha,2+\alpha}$, $f(T) \not\equiv 0$.  Moreover, there is a constant $C_{\Omega,p}>0$, which depends only on $\Omega$ and $p$ such that if $f$ satisfies
\begin{equation}\label{ignition_smallness_condition}
	f(T) \leq C_{\Omega,p}(T - \theta_0)_+^p,
\end{equation}
with $p > 2$, then the left limit is one.  In other words
\[
	\lim_{x\to-\infty} T(t,x,z) = 1.
\]
\end{thm}
The assumption \eqref{ignition_smallness_condition} is made for purely
technical reasons and is similar to that in
~\cite{BCR, CLR, Lew, LM}.

The general idea of the proof is to marry the approaches
from~\cite{BCR} and~\cite{BH_Period, Xin}.  However, significant
difficulties arise from having an unknown fluid flow, unlike
in~\cite{BH_Period}, and from having a non-elliptic operator after
changing variables to the moving from, unlike in~\cite{BCR}.  The
proof proceeds as follows.  First, we consider a finite domain,
$[-a,a]\times\Omega_p$, and examine a regularized version of the
problem.  The operator $L$ is not elliptic so we add a second order
term with an $\epsilon$ weight.  In addition, we smooth the vorticity
by convolution in the equation relating it to the fluid velocity.
Smoothing the vorticity is novel and does not appear in the analysis of
previous works.  It provides the necessary regularity for $u$ in order
to apply the main result from~\cite{BH_Gradient} when taking the limit
$\epsilon \to 0$.

In Section \ref{problem_in_a_finite_domain}, we obtain a priori
estimates on the finite domain so that we can apply a fixed point
theorem to find a solution to this related problem.  The main
difference between our work in this section and the analogous work
in~\cite{BCR} is that we use a regularized version of the stream
function formulation for the fluid velocity.  This allows us to get
the requisite estimates to pass to the infinite cylinder.  These
estimates are contained in Proposition \ref{finite_dom_bounds}.

In Section \ref{solutions_on_the_infinite_cylinder}, we will use these
bounds to take the limit $a\to\infty$.  In order to take the limit
$\epsilon\to 0$, we will obtain new estimates independent of
$\epsilon$.  The main difficulties in this section are the lower bound
on the front speed in Proposition \ref{pos_front_speed_unregularized}
and the upper bound on the fluid velocity.  The proof of the lower
bound, while similar in spirit to the lower bound in~\cite{BH_Period},
is  more difficult because of the lack of monotonicity results
used in that paper.  The upper bound on the fluid velocity involves
new estimates on a family of degenerate elliptic equations.

Finally, in Section \ref{solutions_of_the_unregularized_equation}, we will use the bounds from Section \ref{solutions_on_the_infinite_cylinder} along with various parabolic and elliptic regularity results to obtain new estimates on our functions.  This will allow us to conclude the proof of Theorem \ref{the_theorem} by de-convolving the vorticity and thereby obtaining a solution to \eqref{the_PDE}.

{\bf Acknowledgements:} I would like to thank Lenya Ryzhik for
suggesting the project, and Francois Hamel for helpful discussions.

%%%%%%%%%%%%%%%%%%%%%%%%%%%%%%%%%%%%%
%%%%%%%%%%%%%%%%%%%%%%%%%%%%%%%%%%%%%
%%%                               %%%
%%%    PROBLEM IN A FINITE DOMAIN %%%
%%%                               %%%
%%%%%%%%%%%%%%%%%%%%%%%%%%%%%%%%%%%%%
%%%%%%%%%%%%%%%%%%%%%%%%%%%%%%%%%%%%%

\section{The Problem in a Finite Domain}\label{problem_in_a_finite_domain}

We begin by considering this problem on the domain $R_a = [-a, a]
\times \Omega_p$ and with the elliptic operator $L_\epsilon = L +
\epsilon \partial^2_s$ where $L$ is the linear operator defined in
\eqref{definition_of_operator} and where $0<\epsilon < 1/2$.  This regularization method was first used in~\cite{Xin}.  Also, following the development in ~\cite{CLR}, in order to avoid regularity issues, we let $D_a$ be a
smooth domain such that $R_{a+2\ell} \subsetneq D_a \subsetneq
R_{a+3\ell}$.  We define the ``ends'' of $D_a$ as
\[
E = \{(s,x,z)\in \partial D_a: (x,z) \notin P, s \geq a + 2\ell \text{ or } s \leq -a - 2 \ell\}.
\]
\begin{figure}
\centering
\subfloat{\includegraphics[width=.4\textwidth]{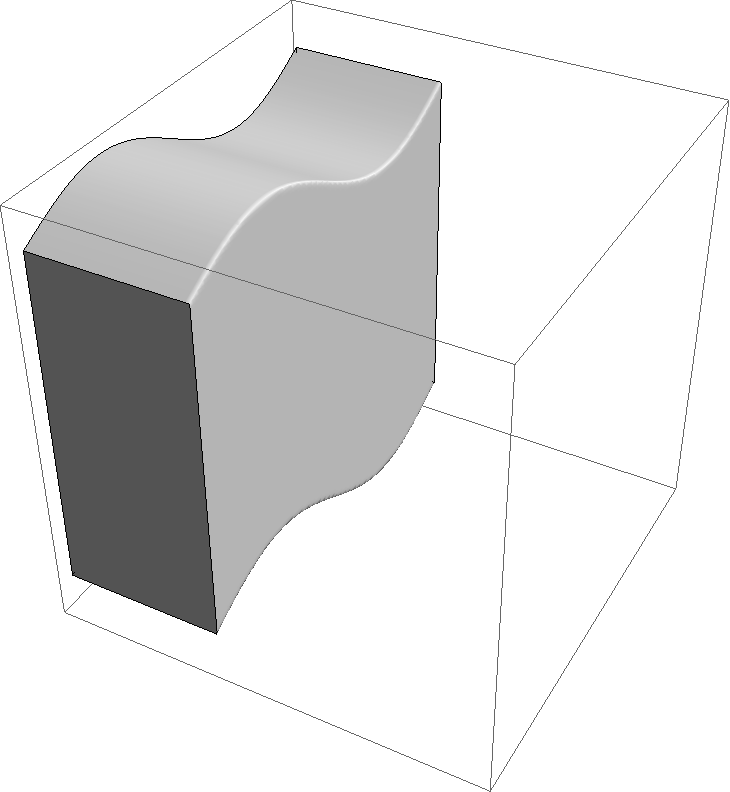}}\subfloat{\includegraphics[width=.4\textwidth]{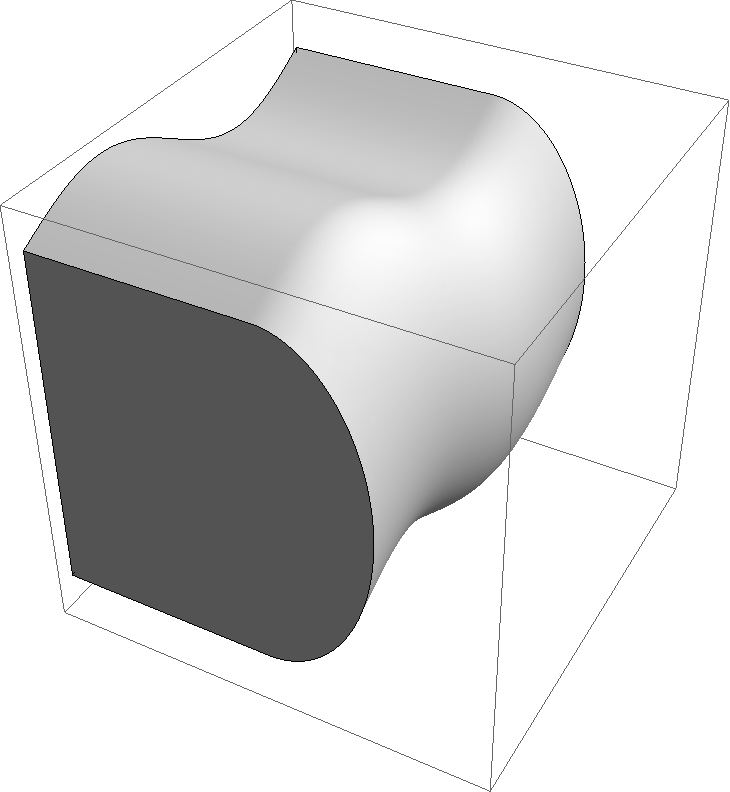}}
\caption*{The right end of the domain $R_a$.\qquad\qquad  The right end of the domain $D_a$.}
\end{figure}

We also introduce some new notation.  Let $\varphi$ be a non-negative
function in $C_c^\infty(\R^3)$ with $\|\varphi\|_{L^1(\R^3)}=1$.
Define $\varphi_\delta(s,x,z) = \delta^{-3} \varphi(s/\delta, x/\delta,
z/\delta)$.  If $g$ is a function on $D_a$ which satisfies periodic
boundary conditions on $P$ and which vanishes on the rest of the
boundary, we can extend it to all of $\R^3$ such that the $H^1$ and
$C^1$ norms increase by a factor of at most 2.  Then we define
\begin{equation*}
	\tilde{g}_\delta (s,x,z) = (f*\varphi_\delta)(s,x,z) = \frac{1}{\delta^3}\int g(s-s', x-x',z-z')\varphi(s',x',z')ds'dx'dz'.
\end{equation*}
We will use this to smooth the vorticity in the first few sections of this paper.  This introduces $\delta$-dependence which propagates to every function.  Hence, decorating functions with a $\delta$ subscript becomes ambiguous since one cannot tell if the subscript indicates that the function has been convolved or simply solves a PDE with a coefficient which depends on $\delta$.   The tilde notation is meant to clear up this ambiguity.

%\begin{equation*}\label{convolved_flow}
%	\begin{split}
%		\tilde{f}_\delta(s,x,z) &= (F * \varphi_\delta)(x-s,x,z)\\
%			&= \int F(x-s-s',x-x',z-z')\varphi_\delta(s',x',z')ds'dx'dz'\\
%			&=\int f(s+s'-x',x-x',z-z')\varphi_\delta(s',x',z')ds'dx'dz'.
%	\end{split}
%\end{equation*}
Since convolutions are
continuous as maps from Sobolev spaces to themselves and H\"older
spaces to themselves, any bounds we obtain for $g$ will carry bounds
on $\tilde{g}_\delta$.

We consider a regularized problem, which will, in the limit as
$\delta$ and $\epsilon$ tend to zero and as $a$ tends to infinity,
converge to the solution of the pulsating front problem.  For ease of notation, when no
confusion will arise, we'll suppress the dependence of the functions
on $a$, $\delta$, and $\epsilon$.  Moreover, in this section, we will wish to prove existence by making use of the Leray-Shauder degree theory.  Hence, we will add a $\tau \in [0,1]$ into our equations.  This gives us the problem below:
\begin{align} \label{eq_finiteprob1}
	\begin{split}
		-cT^m_s + u^m \cdot \tn T^m &= L_\epsilon T^m + \tau f(T^m) \mbox{ on } R_a\\
		-c\omega^m_s - L_\epsilon \omega^m &= \tau\hat{e}\cdot\tn^\perp T^m \mbox{ on } D_a\\
		L_\epsilon\Psi^m &= \tilde{\omega}^m \mbox{ on } D_a\\
		u^m &= (\Psi^m_z, -\Psi^m_s - \Psi^m_x),
	\end{split}
\end{align}
with the following boundary conditions:
\begin{equation}\label{eq_finiteprob2}
	\begin{split}
		\eta \cdot \tn T^m &= 0 \quad\mbox{ on } [-a,a]\times B\\
		T^m(-a, x, z) &= 1\\
		T^m(a,x,z) &= 0\\
		\Psi^m  &= 0 \quad\mbox{ on  $([-a-2\ell,a+2\ell]\times B)\cup E$}\\
		\omega^m &= 0 \quad\mbox{ on } ([-a-2\ell,a+2\ell]\times B)\cup E.\\
	\end{split}
\end{equation}
In addition, all functions satisfy periodic boundary conditions on boundary portion $[-a,a]\times P$.  In order to make \eqref{eq_finiteprob1} well-defined, we extend the function $T^m$ from $R_a$ to $D_a$ by reflection as in~\cite{CLR}.  Explicitly, we write
\begin{equation*}
T^m(s,x,z) =
	\left\{
		\begin{array}{ll}
			-T^m(-2a - s, x, z) + 2T^m(-a,x,z) \qquad &\mbox{ if } s < -a\\
			-T^m(2a - s) + 2T^m(a,x,z) \qquad &\mbox{ if } s > a.
		\end{array} \right.
\end{equation*}
Notice that the extension does not increase the $C^{1,\alpha}$ norm of $T^m$, up to a multiplicative factor.

We wish to show, using the Leray-Schauder fixed point theorem, that we can find
$T^m$ which is bounded uniformly in $a$ in $C^{2,\alpha}$ (but not
necessarily in $\epsilon$ or $\delta$) and which solves the first
equation in \eqref{eq_finiteprob1}, with $\omega^m$ which solves the
second equation in $C^{2,\alpha}$.  On the other hand, $\Psi^m$ will
solve the third equation weakly but will be smooth regularity on
$R_a$.

In order to find a solution we will prove the following a priori
bounds.
\begin{prop}\label{finite_dom_bounds}
  There exists a real number $a_0$ and a positive constant $C =
  C(\epsilon, \delta)$ such that if $T^m$ is a solution to problem
  \eqref{eq_finiteprob1}-\eqref{eq_finiteprob2} and if $a \geq a_0$,
  then
\begin{equation*}
|c| + \|T^m\|_{C^{2,\alpha}} + \|\Psi^m\|_{C^{3,\alpha}} + \|\omega^m\|_{C^{2,\alpha}} + \|\nabla T^m\|_{L^2} + \|\Psi^m\|_{H^{3}} + \|\omega^m\|_{H^2} < C.
\end{equation*}
Here all bounds hold on $R_a$.
\end{prop}

In order to do this we will get a series of bounds which will eventually close.  We will start by proving a relationship between the front speed, $c$, and the flow velocity, $u^m$.  Then we will get a series of $L^2$ and $H^1$ bounds.  First we will show a bound of $\omega^m$ in terms of $T^m$.  Then we will get a bound of $u^m$ in terms of $\omega^m$.  Finally, we will get a bound of $T^m$ in terms of $c$.  Putting these all together we end up with an inequality where the same norm of $T^m$ shows up on both sides, though with a larger exponent on the left than on the right, effectively finishing the proposition above.

Virtually all the bounds in this section will depend on $\epsilon$ and
$\delta$ unless stated otherwise.  On the other hand, the bounds will
be independent of $a$, allowing us to take local limits as $a$ tends
to infinity.  In the following sections, we will obtain new bounds
allowing us to take limits as $\epsilon$ and then $\delta$ tend to
zero.

\subsubsection*{A bound on the front speed by fluid velocity}

We start by proving a bound on $c$.  This proof is essentially the
same as that of~\cite{BCR}.

%%%%%%%%%%%%%%%%%%%%%%%%%%
% FRONT SPEED AND L^OO BOUND OF U

\begin{lem}\label{front_speed}
  There exists $a_0>0$ and $C_0>0$ so that if $(c,T^m, u^m)$ solves
  \eqref{eq_finiteprob1}-\eqref{eq_finiteprob2} then for all $a
  \geq a_0$ and all $\epsilon\in(0,1/4)$ we have 
$$-C_0(1 +\|u^m\|_{L^\infty}) \leq c \leq C_0(1+\|u^m\|_{L^\infty}).$$
\end{lem}
\begin{proof} Let $\psi_e$ be the principal eigenfunction (normalized
to have $L^\infty$ norm 1) of the operator 
\[
-\Delta\psi_e + 2 \partial_x\psi_e=\mu_e\psi_e,~~\hbox{$\psi_e>0$ 
 }
\]
defined in $\Omega_p$, with the boundary conditions
\[
-\eta_1 \psi_e +
  \eta\cdot \nabla\psi_e = 0\hbox{ on } B,
\]
and periodic boundary conditions in $x$ (that is, at $x=0$ and
$x=\ell$).  As in~\cite{BH_Period}, Proposition~5.7, such an
eigenfunction exists and is positive on $\overline{\Omega_p}$.
Moreover, we can see by multiplying the equation by $\psi_e$ and
integrating that the principal eigenvalue  $\mu_e$ is
positive.  

Define the function 
\[
\gamma_A(s) = Ae^{-(s+a)}\psi_e(x,z).
\]
Notice that $\gamma_A$ satisfies the same boundary conditions as $T^m$
on $[-a,a]\times B$ and $[-a,a]\times P$ but not at $s=-a$ or $s=a$.  Then
\begin{equation*}
	\begin{split}
		-c(\gamma_A)_s + u^m\cdot \tn \gamma_A - L_\epsilon \gamma_A
			&= Ae^{-(s+a)}[c\psi_e - u_1^m \psi_e + u^m \cdot \nabla \psi_e - (1+\epsilon)\psi_e  - \Delta\psi_e + 2(\psi_e)_x]\\
			&= Ae^{-(s+a)}[c \psi_e - u_1^m \psi_e + u^m\cdot
                        \nabla \psi_e  
-(1+\epsilon)\psi_e + \mu_e \psi_e]\\
			&\geq Ae^{-(s+a)}[(c - \|u_1^m\|_{L^\infty} - (1+\epsilon)) \psi_e - \|u^m\|_{L^\infty} \|\nabla \psi_e\|_{L^\infty})]\\
			&\geq M A\psi_e e^{-(s+a)} = M\gamma_A\ge \tau f(\gamma_A),
	\end{split}
\end{equation*}
if
\begin{equation}\label{supersolution_condition}
c \geq M + \|u^m\|_{L^\infty}\left(1+\left\|\frac{\nabla \psi_e}{\inf \psi_e}\right\|_{L^\infty}\right) + (1+\epsilon).
\end{equation}
We defined here
\[
M=\sup_{T\in(0,1)}\frac{f(T)}{T}.
\]
Hence, $\gamma_A$ is a super solution for all $A>0$.  If we take $A$
larger than $e^{2a}/(\inf \psi_e)$, then $\gamma_A > 1 \geq T^m$
everywhere in $R_a$.  Hence, we can define 
\[
A_0 = \inf \{A: \gamma \geq T^m\hbox{ everywhere in $R_a$}\}.
\]
It follows from compactness of $R_a$ and the continuity of $T^m$ and $\gamma_A$  that there is a
point $(s_0,x_0,z_0)$ such that $T^m(s_0,x_0,z_0) =
\gamma_{A_0}(s_0,x_0,z_0)$. The strong maximum
principle implies that this cannot be an interior point or a point on the
periodic boundary of $\Omega_p$.  In addition, the Hopf maximum principle
implies that it cannot happen on the boundary $[-a,a]\times B$.  Moreover since 
$$T^m(a,x,z) = 0 < (\inf \psi_e) e^{-2a} \leq \gamma(a,x,z),$$
it cannot be that $s_0 = a$.  Hence, it must be that $s_0 = -a$, from
which it follows that $A_0 \leq (\inf \psi_e)^{-1}$.

On the other hand, the normalization condition \eqref{normalization},
implies then that $\theta_0 \leq A_0 e^{-a}$.  This clearly doesn't hold
for $a$ sufficiently large, so it must be that
\eqref{supersolution_condition} cannot hold.  In other words, we must
have that
$$
c \leq 1 + \epsilon + C\|u\|_{L^\infty} + M,
$$
as desired.

Now, we prove the lower bound.  Let $\psi_e'$ be the principal
eigenfunction of the operator 
\[
-\Delta\psi_e'-2\partial_x\psi_e'=\mu_e'\psi_e',
\]
in $\Omega_p$, with periodic
boundary conditions in $x$ and satisfying 
\[
\eta_1 \psi_e + \eta \cdot
\nabla \psi_e\hbox{ on } B.
\]  
Again, we know that $\psi_e'$ is positive, we
take it to have $L^\infty$ norm 1, and that  
$\mu_e'$ is positive.  Let 
\[
\gamma_A(s,x,z) = 1 - Ae^{s-a}\psi_e'(x,z).
\]
Arguing as before, we see that this is a subsolution if $A>0$ and
\begin{equation}\label{subsolution_condition}
	c
		\leq -\|u\|_{L^\infty}\left(1 + \left\|\frac{\nabla\psi_e'}{\psi_e'}\right\|_{L^\infty}\right) - (1 + \epsilon).
\end{equation}

Moreover, as before, if (\ref{subsolution_condition}) holds, 
we can show that we can take $A$ to be at least as small as one with $\gamma_A \leq T^m$.  Then we have that, by the normalization condition
\eqref{normalization},
$$
1 - e^{-a}
\leq \gamma_1(0,x,z) \leq T^m(0,x,z) \leq \theta_0.
$$ 
This leads to a
contradiction if $a$ is large enough.  Hence it must be that
\eqref{subsolution_condition} does not hold when $a$ is large.  This
implies that our lower bound for $c$ holds.
\end{proof}

Now we will begin collecting $L^2$ bounds on various objects.  The
goal here is to eventually close all these bounds, which leads to the
proof of Proposition \ref{finite_dom_bounds}.  

\subsubsection*{A bound on vorticity by temperature}

The first is the
bound  on the vorticity.

%%%%%
%  Omega H^2 Bound
%%%%%
\begin{prop}\label{omega_h2}
  There exists a constant $C = C(\epsilon)>0$ so that if $T^m$ and
  $\omega^m$ satisfy \eqref{eq_finiteprob1}-\eqref{eq_finiteprob2},
  then
\begin{equation*}
	\|\omega^m\|_{H^1(D_a)} \leq C \|\tn T^m\|_{L^2(R_a)}.
\end{equation*}
\end{prop}
\begin{proof}
  We multiply the second equation in \eqref{eq_finiteprob1} by
  $\omega^m$ and integrate in all three coordinates.

\[
	-c \int_{D_a} \omega^m \omega^m_s dsdxdz - \int_{D_a} \omega^m L_\epsilon \omega^m dsdxdz
		= \tau \int_{D_a} \omega^m (\hat{e}\cdot \tn T^m)dsdxdz
\]
The first term vanishes because $\omega^m=0$ on the boundary.
Integration by parts (along with the boundary conditions for
$\omega^m$), H\"older's inequality, and an application of the
Poincar\'e inequality in the $z$ variable gives us
\begin{equation*}
	\begin{split}
	\int_{D_a} (|\tn \omega^m|^2 + \epsilon (\omega^m_s)^2)dsdxdz
		&\leq C\|\omega^m_z\|_{L^2(_{D_a})} \cdot \|\tn T^m\|_{L^2(_{D_a})} \\
		&\leq C\left(\int_{D_a} (|\tn \omega^m|^2 + \epsilon (\omega^m_s)^2 )dsdxdz\right)^{1/2}\|\tn T^m\|_{L^2(D_a)}
	\end{split}
\end{equation*}
Notice that the Poincar\'e Inequality can be applied on cross-sections with fixed $x$ and $s$ to give us that $\|\omega^m\| \leq C \|\omega^m_z\|$.  Hence $C$ doesn't depend on $a$.  Then we arrive at
\begin{equation}
\|\omega^m\|_{H^1(D_a)} \leq C \|\tn T^m\|_{L^2(R_a)},
\end{equation}
and this finishes the proof.
\end{proof}

It follows from the properties of convolutions and the change of
variables that we made that this gives us an $H^1$ bound on
$\tilde{\omega}^m$, with an added constant independent of $a$,
$\epsilon$, and $\delta$.  In addition, we get $H^k$ bounds on
$\tilde{\omega}^m$ for any $k>1$, though the constants in these will
have some $\delta$ dependence.

%%%%
%  Bounds on \Psi^m
%%%%
%%%%%%%
% L oo bound on u
%%%%%%%

\subsubsection*{A bound on velocity by vorticity}

We now bound the fluid velocity by vorticity.  To do this, we will use the fact that convolution with a smooth, compactly supported function is a bounded operator from $L^2$ to $H^k$ for any $k$.

\begin{lem}\label{finite_domain_sobolev_stream_bounds}
For any $k$, there exists a constant $C = C(\epsilon, \delta)>0$ such that if $(c,T^m, u^m)$ solves \eqref{eq_finiteprob1}-\eqref{eq_finiteprob2} then
\[
\|\Psi^m \|_{H^k(D_a)} \leq C \| \omega^m \|_{H^1(D_a)}.
\]
In addition, the Sobolev embedding theorem gives us that
\[
\|u^m\|_{C^k(D_a)} \leq C \| \omega^m \|_{H^1(D_a)}.
\]
\end{lem}
\begin{proof}
First we will obtain an $L^2$ bound on $\Psi^m$ by multiplying equation \eqref{eq_finiteprob1} by $\Psi^m$, integrating by parts, and using the boundary conditions of $\Psi^m$, we obtain
\[
\int_{D_a} \left[ \epsilon |\Psi^m_s|^2 + |\nabla \Psi^m|^2 \right] dxdzds
	= -\int_{D_a} \Psi^m \tilde\omega^m dxdzds.
\]
Then the using the Cauchy-Schwarz and Poincar\'e inequalities we obtain
\[
\|\Psi^m\|_{L^2}^2
	\leq C\int_{D_a} \left[ \epsilon |\Psi^m_s|^2 + |\nabla \Psi^m|^2 \right] dxdzds 
	\leq C\|\Psi^m\|_{L^2} \|\tilde\omega^m\|_{L^2},
\]
where the constant $C$ depends only on the domain $D$.  Hence, we obtain
\[
	\|\Psi^m\|_{L^2} \leq C \|\tilde\omega^m\|_{L^2}.
\]

We will now use the standard elliptic estimates to obtain greater regularity for $\Psi^m$.  In order for these estimates to be useful, we need them to be independent of $a$.  To this end, we argue as in~\cite{BCR}, Lemma 2.3.  In short, we apply the estimates on sets of the form $[x_0, x_0 + 1] \times \partial \Omega_p$ and sum these to obtain bounds independent of $a$.  For any $k\geq 2$, we've summarized this in the following inequality,
\[
\|\Psi^m\|_{H^k}
	\leq C_\epsilon \left(\|\Psi^m\|_{L^2} + \|\tilde\omega^m\|_{H^{k-2}}\right)
	\leq C_\epsilon \|\tilde\omega^m\|_{H^{k-2}}.
\]
The constant in this inequality depends only on $\epsilon$ and $k$.  To finish, we simply use that convolutions are bounded operators, as discussed above.  Hence, we obtain
\[
\|\Psi^m\|_{H^k}
	\leq C_{\epsilon,\delta} \|\omega^m\|_{H^1},
\]
where the constant depends only on $\epsilon$, $\delta$, and $k$.
\end{proof}

\subsubsection*{A bound on temperature by fluid velocity}

\begin{lem}\label{ell_T_bd}
  There exist constants $C>0$ and $a_0 > 0 $ such that if $(c, T^m,
  u^m)$ satisfy \eqref{eq_finiteprob1}-\eqref{eq_finiteprob2} with the
  normalization \eqref{normalization}, and if $a > a_0$, then
\begin{equation}\label{eq1}
	\int_{R_a} |\tn T^m|^2 dxdzds
		+ \epsilon\int_{R_a} (T^m_s)^2 dxdzds
		+ \int_{\Omega_p} T^m_s(a,x,z)dxdz
		\leq C(1 + \|u^m\|_{L^\infty(R_a)}).
\end{equation}
\end{lem}
\begin{proof}
Recall that $T^m$ satisfies the equation
\begin{align}\label{eq2}
	-cT^m_s + u^m \cdot \tn T^m - L_\epsilon T^m = \tau f(T^m)
\end{align}
with the boundary conditions
\begin{align}
	T^m(-a,x,z) = 1,~~ T^m(a, x,z) = 0,~~ 
\eta \cdot \tn T^m(s,x,z) = 0 \mbox{ on } [-a,a]\times B.
\end{align}
Now if we multiply \eqref{eq2} by $(T^m-1)$ and integrate, we obtain
\begin{gather*}
		\frac{-c|\Omega_p|}{2} + \frac{1}{2}\int_{\Omega_p} u_1^m(a,x,z)dxdz + \int_{R_a} |\tn T^m|^2dxdzds + \epsilon\int_{R_a} (T^m_s)^2dxdzds \\
		+ (1+\epsilon) \int_{\Omega_p} T^m_s(a)dxdz = \tau\int_{R_a} (T^m - 1)f(T^m)dxdzds.
\end{gather*}
Notice that the second term is bounded by
$|\Omega_p|\cdot\|u^m\|_{L^\infty}$.  Then, we move the first and
second terms from the left hand side to the right hand side.  Finally
we use Lemma \ref{front_speed} and that $(T^m-1)f(T^m) \leq 0$, which is true by the maximum principle and the definition of $f$, to
obtain
\begin{gather*}
\int_{\Omega_p} T^m_s(a,x,z) dx dz + \epsilon\int_{R_a} (T^m_s)^2dxdzds + \int_{R_a} |\tn T^m|^2 dsdxdz\\
	\leq \frac{c|\Omega_p|}{2} + C\|u^m\|_{L^\infty} \leq C(1 + \|u^m\|_{L^\infty}).
\end{gather*}
\end{proof}

We need one more lemma to close the inequalities we've accumulated so
far.  The proof is essentially the same as in \cite{BCR} with a few
extra terms.

%%%%%%
% Boundary Integral of T Lemma
%%%%%%
\begin{lem} \label{bdry_lemma}
There exists a constant $C>0$ and a constant $a_0$ such that any solution to the system \eqref{eq_finiteprob1} - \eqref{eq_finiteprob2} with the normalization \eqref{normalization}, satisfies, for $a >a_0$
\begin{align}
0 \leq - \int_{\Omega_p} T^m_s(a,x,z)dxdz \leq C\left(1 + \|\tn T^m\|_{L^2(R_a)}\right).
\end{align}
\end{lem}
\begin{proof}
Define the function
\begin{equation}\label{defn_I}
I(s) = \frac{1}{|\Omega_p|}\int_{\Omega_p} T^m(s,x,z) dxdz.
\end{equation}
Then, integrating the equation for $T^m$ in \eqref{eq_finiteprob1} in $x$ and $z$ and using the boundary conditions, gives us
\begin{equation}
	-I_{ss} = G(s),\, I(-a) = 1,\, I(a) = 0.
\end{equation}
Here $G$ is the function given by
\begin{equation}\label{eqn_G}
	G(s) = \frac{1}{(1+\epsilon)|\Omega_p|}\int_{\Omega_p} (\tau f(T^m) - u^m\cdot \tn T^m + c T^m_s) dxdz + \frac{1}{(1+\epsilon)|\Omega_p|}\int_B \eta_1 T^m_s dS,
\end{equation}
where $\eta=(\eta_1, \eta_2)$ is the unit normal to $B$.  We can solve this equation explicitly as
\begin{equation}
I(s) = -\int_{-a}^s (s-r)G(r)dr + As + B
\end{equation}
with constants
\begin{gather*}
A = -\frac{1}{2a} + \frac{1}{2a}\int_{-a}^a (a-r) G(r)dr,\\
B = \frac{1}{2} + \frac{1}{2}\int_{-a}^a (a-r)G(r)dr
\end{gather*}
which we get from the boundary conditions.  Hence, we have that 
\begin{equation}
	I_s(-a) = A, \, I_s(a) = A - \int_{-a}^a G(r)dr.
\end{equation}
In addition, the boundary conditions on $T^m$ give us that $I_s(-a) \leq 0$ and $I_s(a) \leq 0$.  We wish to get lower bounds on these quantities.  Hence, using \eqref{eqn_G}, we have that
\begin{equation}\label{eq_I_right1}
	\begin{split}
	-I_s(a)
		&= \frac{1}{2a} + \frac{1}{2a}\int_{-a}^a (a+r)G(r)dr\\
		&= \frac{1}{2a} + \frac{1}{2a(1+\epsilon)|\Omega_p|}\int_{R_a} (a+r)(\tau f(T^m) - u^m \cdot \tn T^m + cT^m_s)dxdzdr\\
		&+ \frac{1}{2a(1+\epsilon)|\Omega_p|}\int_{B}(a+r) \eta_1 T^m_s dS(x,z,r).
	\end{split}
\end{equation}
In the right side of this equation, we have three terms.  We leave the
first as is and we use the fact that $\eta_1$ is constant in $r$ to
get that the third term is bounded as
\begin{equation}\label{third_term}
	\begin{split}
	\frac{1}{2a|\Omega_p|}\int_{-a}^a \int_{B} (a+r) \eta_1 T^m_s dSdr 
		&= -\frac{1}{2a|\Omega_p|}\int_{B} \eta_1 T^m
                dS(x,z,r)
\leq \frac{|B|}{|\Omega_p|}
		\leq C.
	\end{split}
\end{equation}

The second term is a bit more stubborn.  Integrating by parts and
using the fact that $rf(T^m) \leq 0$ for all $(r,x,z) \in R_a$, gives
us
\begin{equation}\label{second_term}
\begin{split}
\frac{1}{2a|\Omega_p|}& \int_{R_a} (a+r)[\tau f(T^m) - u^m\cdot\tn T^m + cT^m_s] dxdzds \\
	&= \frac{1}{2|\Omega_p|}\int_{R_a} \tau f(T^m)dxdzds 
		+ \frac{1}{2|\Omega_p|} \int_{R_a}
                \frac{r}{a}f(T^m)dxdzds
+ \frac{1}{2a|\Omega_p|} \int_{R_a} u^m_1 T^m dxdzds\\
	&= - \frac{c}{2a|\Omega_p|}\int_{R_a} T^mdxdzds\\\
	&\leq \frac{\tau F}{2} + \|u^m\|_{L^\infty} + |c|.
\end{split}
\end{equation}
Here 
\[
F = |\Omega_p|^{-1}\int f(T^m)dxdzds.
\]
In summary, from \eqref{eq_I_right1}, \eqref{third_term}, and \eqref{second_term},we have
\begin{align}\label{eq_I_right2}
0 \leq -I_s(a) \leq \frac{1}{2a} + \|u^m\|_{L^\infty} + |c| + C + \frac{\tau F}{2}
\end{align}
where $C>0$ is a universal constant given by \eqref{third_term}.

Now we look at $I_s(-a)$.  As before, we have that
\[
	0 \leq -I_s(-a).
\]
Also, we similarly estimate $-I_s(-a)$ from above as follows.
\begin{equation}\label{eq_I_left}
	\begin{split}
	- I_s(-a) &= \frac{1}{2a} - \frac{1}{2a}\int_{-a}^a (a-r)G(r)dr \\
		&=\frac{1}{2a} - \frac{\tau}{2a(1+\epsilon)|\Omega_p|}\int_{R_a} (a-r) f(T^m)dxdzdr\\ 
		&~~+ \frac{1}{2a(1+\epsilon)|\Omega_p|}\int_{R_a} (a-r) u^m \cdot \tn T^m dxdzdr - \frac{c}{2a(1+\epsilon)|\Omega_p|}\int_{R_a} (a-r) T^m_s dxdzdr\\
		&~~- \frac{1}{2a(1+\epsilon)|\Omega_p|} \int_B (a-r)\eta_1 T^m_s dS \\
		&\leq \frac{1}{2a} - \frac{1}{(1+\epsilon)|\Omega_p|} \int_{\Omega_p} u_1^m(-a) dxdzds + \frac{1}{2a(1+\epsilon)|\Omega_p|}\int_{R_a} u^m_1 T^m dxdzds \\
		&~~- \frac{c}{2a(1+\epsilon)|\Omega_p|}\int_{R_a} T^m dxdzds - \frac{1}{2a(1+\epsilon)|\Omega_p|}\int_B \eta_1 T^m dS \\
		&\leq \frac{1}{2a} + 2\|u^m_1\|_{L^\infty} + |c| + C,
	\end{split}
\end{equation}
where $C$ here is a universal constant.  In the computation above, we only used integration by parts and that ${(a-r)f(T^m) \geq 0}$.

There is one more calculation we need to make before we can finish this proof.  Namely, we need to show that the integral of $u\cdot \tn T^m$ vanishes.  To see this, simply integrate by parts as follows
\begin{equation*}
	\int_{R_a} u\cdot \tn T^m dxdzds
		= - \int_{\Omega_p} u_1(-a, x, z) dxdz
		= - \int_{\Omega_p} \Psi^m_z(-a,x,z) dx dz
		= 0.
\end{equation*}
The first equality is a result of the boundary conditions of $T^m$ and
$u^m$.  The last equality is a result of the boundary conditions on
$\Psi^m$.  Hence, when we integrate $G$ from $-a$ to $a$, we get
\[
	\int_{-a}^a G(s)ds
		= \frac{\tau F}{(1+\epsilon)}
			- \frac{1 c}{(1+\epsilon)}
			-\frac{1}{(1+\epsilon)|\Omega_p|}\int_{B\cap\{s = -a\}} \eta_1 dxdz.
\]

To finish the proof we analyze the following three equations:
\begin{gather*}
	I_s(a) - I_s(-a) = -\int_{-a}^a G(s)ds = \frac{1}{(1+\epsilon)}\left(- \tau F + c - \frac{1}{|\Omega_p|}\int_{B\cap\{s = -a\}} \eta_1 dxdz\right)\\
	0 \leq - I_s(a) \leq \frac{1}{2a} + \|u^m\|_{L^\infty} + |c| + C + \frac{\tau F}{2}\\
	0 \leq -I_s(-a) \leq \frac{1}{2a} + 2\|u^m\|_{L^\infty} + |c| + C.
\end{gather*}
From these, we easily get that
\begin{align}\label{eqn_F}
\frac{\tau F}{1+\epsilon} \leq (|c| + C) + \left(\|u^m\|_{L^\infty} + \frac{1}{2a} + |c| + C + \frac{\tau F}{2}\right),
\end{align}
and rearranging this, along with the inequalities from Lemmas \ref{front_speed} and \ref{omega_h2}  and Corollary \ref{finite_domain_sobolev_stream_bounds}, gives us that
\begin{align}
\tau F \leq C (\|\tn T^m\|_{L^2(R_a)} + 1)
\end{align}
for $a>1$.  Hence using this, the inequalities from Lemmas \ref{front_speed} and \ref{omega_h2}, Corollary \ref{finite_domain_sobolev_stream_bounds}, and equation \eqref{eq_I_right2} finishes the proof.
\end{proof}

\subsubsection*{A uniform bound on temperature}

We can now combine the previous estimates to obtain the following proposition.

%%%%%%%
%  Closing of bounds
%%%%%%%

\begin{prop}\label{finite_domain_bound_close}
Let $(c, T^m, u^m)$ satisfy \eqref{eq_finiteprob1} - \eqref{eq_finiteprob2} with the normalization \eqref{normalization}.  Then there exists a constant $C>0$ such that
\begin{equation*}
\|\tn T^m\|_{L^2(R_a)} \leq C.
\end{equation*}
\end{prop}
\begin{proof}
This is simply given by combining Corollary \ref{finite_domain_sobolev_stream_bounds} and Lemmas \ref{ell_T_bd}, \ref{bdry_lemma}, and \ref{omega_h2}.
\end{proof}

Notice that this closes all of our bounds, giving us a uniform bound on $|c|$, $\|u\|_{L^\infty}$ and $\|\omega^m\|_{H^2}$ by using the various estimates from earlier.

\begin{rem}\label{remark_about_finite_dom_bounds} Notice that we can then use these bounds and the usual elliptic theory to give us that $T^m$ is bounded uniformly in $C^{2,\alpha}(R_a)$.  Of course, this in turn gives us that $\omega^m$ is bounded uniformly in $C^{2,\alpha}(D_a)$ as well.  Arguing as before and using Schauder estimates, we can show that $\Psi^m$ is also uniformly bounded in $C^{3,\alpha}(D_a)$.  Hence, we have proven Proposition \ref{finite_dom_bounds}.  These bounds will blow up as $\epsilon$ tends to zero.
\end{rem}

We will make use of the Leray-Schauder topological degree theory, see e.g.~\cite{N}, to prove the existence of a solution to our problem on the finite domain.  Our a priori bounds make this possible.  The proof that follows is similar to that which appears in~\cite{BCR}.

%%%%%
% SOLUTIONS EXIST!! -- FINITE DOMAIN
%%%%%
\begin{prop}
  For each $a$ sufficiently large, there exists a solution to the
  system \eqref{eq_finiteprob1}-\eqref{eq_finiteprob2}.
\end{prop}
\begin{proof}
Let $V = \{(x, T) \in \mathbb{R}\times C^{1,\alpha}(R_a): |x| + \|T\|_{C^{1,\alpha}} \leq C_0 + 1\}$, where $C_0$ is a bound to be chosen later.  Define the operator $S_\tau: \R\times C^{1,\alpha} \to \R\times C^{1,\alpha}$ as $S_\tau(c, Z) = (c + \theta_0 - \max_{s\geq 0} Z(s,x,z), T)$ where $T$ is the solution to
\begin{equation*}
	-cT_s + u \cdot \tn T = L_\epsilon T + \tau f(Z) \mbox{ on } D_a,
\end{equation*}
where $u$ solves
\begin{equation*}
	\begin{split}
		-c\omega - L_\epsilon \omega &= \tau \hat{e} \cdot \tn Z \mbox{ on } D_a\\
		L_\epsilon\Psi &= \tilde{\omega} \mbox{ on } D_a \\
		u &= (\Psi_z, \Psi_x + \Psi_s),
	\end{split}
\end{equation*}
with the boundary conditions as in \eqref{eq_finiteprob2}, and with $Z$ extended as before.  By the usual elliptic theory, $S_\tau$ is continuous and compact.  We wish to find a fixed point of map, $S_1$.  To this end, we will show that the degree of our map is non-zero.  Notice that by our previous work, we can choose $C_0$ large enough so that $\id - S_\tau$ does not vanish on the boundary of $V$.  Then the Leray-Schauder topological degree theory tells us that for all $\tau$,
\[
	\deg(\id - S_\tau, V, 0) = \deg(\id - S_0, V, 0). 
\]
Hence, we need only show that $\deg(\id - S_0, V, 0)$ is non-zero.  Notice that
\[
	\left(\id - S_0\right) (c, Z) = ( \max_{s\geq 0} Z(s,x,z) - \theta_0 , Z - T_0^c),
\]
where $T_0^c$ is the unique solution to
\[
	-c T_s - L_\epsilon T = 0,
\]
with the boundary conditions as in \eqref{eq_finiteprob2}.  Since degree is invariant under homotopy, we will find a map, homotopic to $\id - S_0$, whose degree is easier to compute.  To this end, we notice that $\id - S_0$ is $\Phi_0$ where
\[
	\Phi_\tau(c, Z) = \left(\max_{s\geq 0} Z - \theta_0 , Z - \tau\varphi_c(s) - (1-\tau) T_0^c\right),
\]
and where
\[
	\varphi_c(s) = \frac{e^{-cs} - e^{-ca}}{e^{ca} - e^{-ca}}.
\]

Before we calculate the degree of this map, we must first check that $\Phi_\tau$ does not vanish on the boundary of $V$.  This amounts to obtaining an a priori bound on $c$ independent of $\tau$, since any bound on $c$ also provides an a priori bound of $T_0^c$ by standard elliptic theory.  To this end, suppose that we have $(c,Z)$ is a zero of the map $\Phi_\tau$ for some $\tau$.  Then 
\[
(c,Z) = (c, \tau \varphi_c + (1-\tau) T_0^c), ~\text{ and }~ \max_{s\geq 0} Z = \theta_0.
\]
Choose $C$ large enough that if $c \geq C$ then $\varphi_c(0) \leq \theta_0/3$ and if $-c \geq C$ then $\varphi_c(0) \geq \frac{1+\theta_0}{2}$.  Notice that, by our work in Lemma~\ref{front_speed}, if $c \geq 2$ then $T_0^c \leq A_0 \psi_e e^{-(s+a)}$ for $A_0 = (\inf \psi_e)^{-1}$, and if $c \leq -2$ then $1 - e^{-(s+a)} \leq T_0^c$.  We need to rule out the cases when $|c|$ becomes large.  Let's first check the case when $c \leq -C$.  Here
\[
	\theta_0
	= \max_{s \geq 0} Z(s,x,z)
	\geq \tau \varphi_c(0) + (1-\tau)(1- e^{-a})
	\geq \tau \frac{1 + \theta_0}{2} + (1- \tau) (1 - e^{-a}).
\]
This leads to a contradiction if $a$ is larger than $- \log((1 - \theta_0)/2)$.  Now we check the case when $c \geq C$.  Here
\[
	\theta_0
	=\max_{s \geq 0} Z(s,x,z)
	\leq \tau \varphi_c(0) + (1-\tau) e^{-a}
	\leq \tau \frac{\theta_0}{3} + (1-\tau) A_0 e^{-a}.
\]
This leads to a contradiction if $a$ is larger than $-\log(\theta_0/(3A_0))$.  Hence, in the definition of $V$, if we choose $C_0$ to be larger than $C$, from above, and the bound in Proposition \ref{finite_dom_bounds} then $\id - S_\tau$ and $\Phi_\tau$ do not vanish on the boundary of $V$.

The map $\Phi_1$ is then homotopic to
\[
	\Phi_2(c,Z) = (\varphi_c(0) - \theta_0, Z - \varphi_c).
\]
Let $c_*$ be the unique values such that $\varphi_{c_*}(0) = \theta_0$.  Then $\Phi_2$ is homotopic to
\[
	\Phi_3(c,Z) = \left( \varphi_c(0) - \theta_0 , Z - \varphi_{c_*}\right).
\]
We can calculate the degree of this map.  Indeed, its degree is the product of the degrees of the two component functions.  The first has degree $-1$ since $\varphi_c$ is decreasing in $c$.  The last has degree one.  Hence, $\deg(\id - S_1, V, 0) = -1$, and we conclude that $S_1$ has a fixed point.  This is our desired solution.
\end{proof}

\section{Solutions on the infinite cylinder}\label{solutions_on_the_infinite_cylinder}

Since we have obtained uniform bounds, we can take weak and strong
local limits as $a\to+\infty$, along subsequences if necessary, in the relevant
topologies to get the limits $c_\epsilon$, $T^{m,\epsilon}$,
$\omega^{m,\epsilon}$, $\Psi^{m,\epsilon}$, and $u^{m,\epsilon}$ that
are defined in the infinite cylinder and satisfy the same system of
equations.
That is, $T^{m,\epsilon}$, $\omega^{m,\epsilon}$,
$\Psi^{m,\epsilon}$, and $u^{m,\epsilon}$ are defined on the domain
$\R \times \Omega_p$.
As before, we omit the $\epsilon$ notation whenever there will be no
confusion.  These functions satisfy the system
\begin{equation}\label{regularized_equation_infinite_domain}
	\begin{split}
		-cT^m_s + u^m\cdot \tn T^m &= L_\epsilon T^m + f(T^m) \mbox{ on } \R\times \Omega_p\\
		-c\omega^m_s - L_\epsilon \omega^m &= \hat{e}\cdot \tn T^m \mbox{ on } \R\times \Omega_p\\
		L_\epsilon\Psi^m &= \tilde{\omega}^m  \mbox{ on } \R\times \Omega_p\\
		u &= (\Psi^m_z, \Psi^m_s + \Psi^m_x)
	\end{split}
\end{equation}
with the boundary conditions
\begin{equation}\label{bdry_cond}
	\begin{gathered}
		\eta\cdot \tn T^m = 0,~~ \Psi^m = 0,~~ \omega^m = 0 \quad\mbox{ on } \R\times B\\
		T^m(s,0,z) = T^m(s,\ell, z),~~ \Psi^m(s,0,z) = \Psi^m(s,\ell, z),~~ \omega^m(s,0,z) = \omega^m(s,\ell, z).
	\end{gathered}
\end{equation}
Since $T_a^m$ satisfies the normalization condition
\eqref{normalization} for every $a$ and converges locally uniformly,
the limit $T^m$ must satisfy it as well.  We make one remark about
notation.  We will assume that $a_n$ is a sequence tending to infinity
along which all the relevant functions converge.  We will denote by
$T_n^m$, the function $T_{a_n}^m$, and similarly for $c_n$ and the other
functions.

In order to take the limit   $\epsilon\downarrow 0$, we need to obtain
bounds which are uniform in $\epsilon$.  Our bounds from Section
\ref{problem_in_a_finite_domain} all used the ellipticity of $L_\epsilon$ and,
hence, depend on $\epsilon$.  Our first step  is to show
that $c_\epsilon$ is positive for all $\epsilon$ (though that bound
will also depend on $\epsilon$ first).  This will allow us to prove that $T_n^m$ behaves as we would expect on the right end, namely it and its $s$-derivative vanish.  This will allow us to prove an exponential bound of $T^m$ on the right, which we will need in section 5.  Finally, we will obtain $L^2$ bounds that are uniform in $\epsilon$ on our functions and their derivatives on sets of the form $\R\times K$ where $K\subset \Omega$ is bounded.  We will use these in the next section to take limits.  Finally, we will prove that the front speed, $c_\epsilon$ is bounded away from zero by a constant that depends only on $\delta$.

In this section, some of the constants will depend on $\epsilon$.  We will denote these by a subscript $\epsilon$ since this dependence is important in this section.  Many of the constants will also depend on $\delta$, but we will suppress the notation for now.

Before we begin the proof, notice that since $u^m$ is an element of
both $C^{0,\alpha}$ and $L^2$, it must converge uniformly to $0$
as $s$ tends to infinity.  The same is true for $\Psi^m$, all of its derivatives, all the
first derivatives of $T^m$, $\omega^m$, and all of its
first and second derivatives.

%%%%%
% BURNING LOWER BOUND
%%%%%

\subsubsection*{A lower bound on the burning rate}

Here we will prove that $c_\epsilon$ is positive.  Notice, though, that our lower bound degenerates as we take $\epsilon \downarrow 0$.  We will address this issue in the following section.

\begin{lem}\label{burning_lower_bd}
There is a constant $C_\epsilon>0$ such that
\begin{equation*}
	\int_{R_{a_n}} f(T_n^m)dxdzds \geq C_\epsilon
\end{equation*}
for all $a$ sufficiently large.
\end{lem}
\begin{proof}
The proof is as in \cite{BCR, CKOR}.  We wish to show that there is
a universal constant $C>0$ such that
\begin{equation}\label{reaction_lower}
\left(\int f(T^m_n)dxdzds \right)\left( \int | (T^m_n)_s|^2dxdzds \right) \geq C.
\end{equation}
Since we have a uniform upper bound on $\| (T_n^m)_s\|_2$ given by
Proposition \ref{finite_dom_bounds}, then this inequality will give us the desired bound.

First, notice that we can find
$(x_0,z_0)\in \Omega_p$ such that
$$
\int_{-a_n}^0 |(T_n^m)_s(s,x_0,z_0)|^2 ds
	\leq \frac{3}{|\Omega_p|}\int_{R_{a_n}} | (T_n^m)_s(s,x,z)|^2dxdzds
$$
and
$$\int_{-a_n}^0 f(T_n^m(s,x_0,z_0))ds
	\leq \frac{3}{|\Omega_p|} \int_{R_{a_n}} f(T_n^m)dxdzds.
$$
Then define
$$
s_1 = \inf\{s\in(-a_n, 0): T_n^m(s,x_0,z_0) = r_2\}
$$
and
$$
s_2 = \inf\{s\in(-a_n, 0): T_n^m(s,x_0,z_0) = r_1\}.
$$
Here, $r_1$ and $r_2$ come from Definition \ref{ignition}.  Notice
that $s_2 > s_1$, and notice that both exist by the normalization
condition \eqref{normalization} and the boundary condition on $T_n^m$
at $s = -a_n$.  Then for $s_1 \leq s \leq s_2$ we have that
$f(x_0,z_0,T_n^m) > C$ for some constant $C>0$.  Hence
\begin{equation*}
C |s_2 - s_1|
	\leq \int_{s_1}^{s_2} f(x_0,z_0,T_n^m(s,x_0,z_0))ds
	\leq \frac{3}{|\Omega_p|}\int_{R_{a_n}} f(x,z, T_n^m(s,x,z))dxdzds
\end{equation*}
and
\begin{equation*}
\frac{(1-\theta_0)^2}{4|s_2-s_1|}
	\leq \int_{s_1}^{s_2} | (T_n^m)_s|^2 ds
	\leq \frac{3}{|\Omega_p|}\int_{R_{a_n}} |(T_n^m)_s|^2dxdzds.
\end{equation*}
Multiplying these two inequalities gives us the desired inequality, \eqref{reaction_lower}.
\end{proof}

%%%%%
%%%%%

\subsubsection*{Positivity of the speed and behavior of the
  temperature on the right}

We will use the bound on the burning rate to get a lower bound on the speed of the front.  This will be crucial in showing that our solutions are non-trivial up to this point.  It will also allow us to make a meaningful change of variables back to the stationary frame where we wish to ultimately show a solution exists.  The next two proofs are similar to those found in \cite{BCR}.

\begin{lem}\label{pos_front_speed}
The front speed, $c_\epsilon$, of the solution obtained above is strictly positive.
\end{lem}
\begin{proof}
We first obtain an equality for use later.  Define the function
  $\Phi_n(s,x,z) = T_n^m(s+a_n,x,z)$ on the set $[-2a_n, 0]\times
  \Omega_p$.  Similarly, define $U_n(s,x,z) = u_n^m(s+a_n,x,z)$, and
  let $V$ and $W$ be the component functions of $U$, i.e. let $U =
  (V,W)$.  Using our a priori bounds for $T_n^m$ and $u_n^m$, we can
  take the limit, along a subsequence if necessary, to obtain
  functions $\Phi$ and $U$ which satisfy the equation
\begin{equation}\label{recenter_bdry}
	-c_\epsilon \Phi_s + U\cdot \tn \Phi = L_\epsilon \Phi
\end{equation}
on the set $(-\infty, 0]\times\Omega_p$.  Moreover, $\Phi(0,x,z) = 0$ for all $(x,z) \in \Omega_p$.  By our choice of translation, we get that $0 \leq \Phi\leq \theta_0$.  Now, integrating \eqref{recenter_bdry} and taking limits as $s$ tends to $-\infty$ gives us
\begin{equation}\label{bdry_deriv_conc}
	\begin{split}
		c_\epsilon \lim_{s\to-\infty}\int_{\Omega_p} \Phi(s,x,z)dxdz
			=(1+\epsilon) \int_{\Omega_p} \Phi_s(0,x,z)dxdz \leq 0.
	\end{split}
\end{equation}

Now we will proceed with the proof.  Notice that by integrating the equation for $T_n$ we arrive at the following
\begin{equation*}
	\begin{split}
		c_n |\Omega_p| - &\int_{\Omega_p} u_1^m(-a_n, x, z) dx dz\\
			&= (1+\epsilon)\int_{\Omega_p}\left[\frac{\partial T_n^m}{\partial s}(a_n,x,z) - \frac{\partial T_n^m}{\partial s}(-a_n,x,z)\right]dxdz + \int f(T_n^m) dx dz ds \\
			&\geq (1+\epsilon)\int_{\Omega_p}\frac{\partial T_n^m}{\partial s}(a_n,x,z) + \int f(T_n^m)dxdzds
	\end{split}
\end{equation*}
The inequality follows from the non-positivity of $\frac{\partial T_n^m}{\partial s}(-a_n,x,z)$, since $T_n^m(-a_n,x,z) = 1$.  First, we show that the second term on the first line is equal to zero.  To see this, notice that $u_1^m = \Psi^m_z$.  Hence we simply integrate to get
\begin{equation*}
	\int_{\Omega_p} u_1^m(-a_n,x,z)dxdz
		= \int_{\partial\Omega_p} \eta_2 \Psi^m(-a_n,x,z)dS(x,z)
		= 0
\end{equation*}
The last equality follows from the boundary conditions for $\Psi^m$.  Hence we arrive at
\begin{equation}\label{front_sp_pos_1}
	c |\Omega_p|
		\geq (1+\epsilon) \int_{\Omega_p} \frac{\partial T_n^m}{\partial s}(a_n,x,z)dxdz + \int f(T^m)dxdz.
\end{equation}
Taking the limit as $n$ tends to infinity in \eqref{front_sp_pos_1} and using the equation \eqref{bdry_deriv_conc}, we arrive at
\begin{equation*}
\begin{split}
c_\epsilon|\Omega_p|
&\geq c_\epsilon  (1 + \epsilon) \lim_{s\to-\infty}\int_{\Omega_p} \Phi(s,x,z)dxdz + \int f(T^m)dxdzds\\
&\geq -|c_\epsilon|(1+\epsilon)|\Omega_p| \theta_0 + \int f(T^m)dxdzds.
\end{split}
\end{equation*}
If $\epsilon>0$ is small enough, then $(1+\epsilon) \Phi^- \leq (1+\epsilon)\theta_0 < 1$ and hence, using Lemma \ref{burning_lower_bd} we arrive at
\begin{equation}
c_\epsilon \geq \frac{C_\epsilon}{|\Omega_p|(1 - (1+\epsilon)\theta_0)} > 0.
\end{equation}
This completes the proof of the lemma.
\end{proof}

%%%%%
%%%%%

\begin{lem}\label{bdry_deriv_dies}
There exists a sequence $a_n \to \infty$ such that
\begin{equation*}
	\lim_{n\to\infty}\left| \frac{\partial T_n^m}{\partial s}(a_n, x, z)\right| = 0
\end{equation*}
uniformly in $x$ and $z$.  Moreover, we have
\[
\lim_{n\to\infty} T_n(a_n - s_0, x, z) = 0
\]
for all $s_0\in \R$.
\end{lem}
\begin{proof}
Recall the function $\Phi$ from Lemma \ref{pos_front_speed}.  We will show that $\Phi(s,x,z)$ converges to a constant as $s\to
-\infty$.  To see this, assume we have two sequences $(s_k,
x_k, z_k)$ and $(s_k',x_k',z_k')$ such that $s_k < s_k' < s_{k+1}$,
$s_k$ and $s_k'$ tend to $-\infty$, and $T^m(s_k,x_k,z_k)$ converges
to $\theta$ and $T^m(s_k',x_k',z_k')$ converges to $\theta'$.  Then we
integrate \eqref{recenter_bdry} over $[s_k,s_k']\times \Omega_p$ to
obtain
\begin{gather*}
c_\epsilon\int_{\Omega_p} (\Phi(s_k',x,z) - \Phi(s_k,x,z))dxdz
	+ \int_{\Omega_p}(V(s_k',x,z)\Phi(s_k',x,z) - V(s_k,x,z)\Phi(s_k,x,z))dxdx\\
	= (1+\epsilon)\int_{\Omega_p} (\Phi_s(s_k',x,z) - \Phi_s(s_k,x,z))dxdz
\end{gather*}
Notice that both $U$ and $\nabla_{x,z} \Phi$ tend uniformly to zero as $s\to-\infty$ since both are uniformly bounded in $L^2$ and in $C^{0,\alpha}$.  Hence taking the limit as $k$ tends to infinity, we obtain
\begin{gather*}
c_\epsilon|\Omega_p|(\theta' - \theta) = 0.
\end{gather*}
Hence, since $c_\epsilon > 0$, it follows that $\Phi$ converges to a constant on the left, call it $\Phi^-$.  Integrating equation \eqref{recenter_bdry} gives us that
\[
	c_\epsilon |\Omega_p| \Phi^- = 0
\]
Hence $\Phi^-=0$.  Then the maximum principle assures us that $\Phi \equiv 0$, finishing the proof.
\end{proof}

\subsubsection*{The fluid velocity on the right}

Here we will show that the fluid velocity on the finite cylinder, $u_n^m$, converges to zero uniformly in $n$ as $s \to \infty$.  This is necessary for proving that the temperature, $T_n^m$, has the same behavior.

%%%%%%
%%%%%%
\begin{lem}\label{flow_tends_to_zero}
For each $\mu, \delta, \epsilon>0$, there exists $R_\epsilon = R(\mu,\delta,\epsilon)<\infty$ such that, for all $n$ and for all $s \geq R_\epsilon$, we have that $|u_n^m(s,x,z)| \leq \mu$ for all $(s,x,z) \in [R_\epsilon, a_n]\times\Omega_p$.
\end{lem}
\begin{proof}
We argue by contradiction.  Suppose there is some $\mu>0$ and some sequence $(s_n, x_n, z_n) \in [-a_n, a_n]\times\Omega_p$, with $s_n$ tending to infinity, such that $|u_n^m(s_n, x_n, z_n)| \geq \mu$.  Define the recentered functions $\Phi_n(s, x, z) = T_n^m(s + s_n, x, z)$, $W_n(s, x, z) = \omega_n^m(s + s_n, x, z)$, $U_n(s, x, z) = u^{n, m}(s + s_n, x, z)$ and $S_n(s,x,z) = \Psi_n^m(s+r_n, x, z)$ on $[-a_n-s_n, a_n - s_n]\times \Omega_p$.  There are two cases.

\textit{Case 1:} $a_n - s_n \to \infty$.  Since all the recentered functions satisfy the same bounds as the usual functions, we can take limits in all the relevant topologies, to get functions $\Phi$, $W$, $U$, and $S$, which satisfy the following equations on $\R\times \Omega_p$
\begin{equation}\label{recentered_flow2zero}
	\begin{split}
		-c_\epsilon \Phi_s + u \cdot \tn \Phi &= L_\epsilon \Phi \\
		-c_\epsilon W_s + L_\epsilon W &= \hat{e} \cdot \tn \Phi \\
		L_\epsilon S &= \widetilde{W} \\
		U &= (S_z, S_x + S_s)
	\end{split}
\end{equation}
with the usual boundary conditions.  The first equation is linear since, by our choice of $s_n$, we have that $0 \leq \Phi \leq \theta_0$.  Similarly as in Lemma \ref{bdry_deriv_dies}, we can show that as $s$ tends to $-\infty$ and $\infty$, $\Phi(s,x,z)$ tends to $\Phi^-$ and $\Phi^+$, respectively, uniformly in $x$ and $z$.  Hence integrating the first equation in \eqref{recentered_flow2zero}, we arrive at
$$
c_\epsilon |\Omega_p|(\Phi^- - \Phi^+) = 0.
$$
Since $c_\epsilon>0$ by Lemma \ref{pos_front_speed}, then we have that the $\Phi^- = \Phi^+$.  By the maximum principle, we have that $\Phi$ is constant.  As a result of this, the Dirichlet boundary conditions, and equation \eqref{recentered_flow2zero}, we have that $W$ and $S$ must be zero.  Hence $U$ is equal to $0$.  However, since $U_n$ converges locally to $U$ in $C^{0,\alpha}$ and $\max_{x,z} U_n(0,x,z) \geq \delta$ for every $n$, then $U(0,x,z) \geq \delta>0$ for some $(x,z)\in \Omega_p$ as well.  This is a contradiction.

\textit{Case 2:} $a_n - s_n \to b \in[0,\infty)$.  By Lemma \ref{bdry_deriv_dies}, $\Phi \equiv 0$.  Hence, we argue as above to conclude that $U \equiv 0$.  This is a contradiction, as before.
\end{proof}

%%%%%
%%%%%
\begin{cor}\label{reg_right_limit}
For every $\epsilon$ sufficiently small, every $n$ large enough, and $0<\alpha \leq c_\epsilon/8$, there are constants $C_\alpha, R_\epsilon>0$, which depend only on $\alpha$ and $\epsilon$, respectively, such that
$$
	T^m_n(s,x,z) \leq C_\alpha e^{- \alpha(s - R_\epsilon)}. 
$$
As a result
\[
	T^{m,\epsilon} \leq C_\alpha e^{-\alpha(s-R_\alpha)}.
\]
\end{cor}
\begin{proof}
  Recall that $\epsilon < 1/2$.  Choose $n$ large enough that
  $|c_{\epsilon} - c_{\epsilon, a_n}| < c_{\epsilon}/2$.  Then, as we
  did in the proof of Lemma~\ref{front_speed}, let $\psi_e$ be the
  principle eigenfunction of the operator $-\Delta + 2
  \alpha \partial_x$ on $\Omega_p$ with periodic boundary conditions
  on $P$ and satisfying $-\alpha \eta_1\psi_e + \eta \cdot \nabla
  \psi_e$ on $B$.  Then we
  choose $R$ such that
$$\|u^m\|_{L^\infty([R, a_n]\times \Omega_p)}
	\leq \min\left\{\frac{c_\epsilon}{16}, \frac{\alpha c_\epsilon}{16 \|\frac{\nabla \psi_e}{\psi_e}\|_{L^\infty}}\right\}.
$$
Then, letting $\gamma_A(s,x,z) = Ae^{-\alpha(s-R_\epsilon)}$, we see that this is a supersolution on $[R,a_n]\times \Omega_p$ since
\begin{equation*}
	\begin{split}
		-c_{\epsilon,a_n} (\gamma_A)_s &+ u\cdot\tn \gamma_A - L_\epsilon \gamma_A\\
			&= Ae^{-\alpha(s-R_\epsilon)}[(c_{\epsilon,a_n, \delta}\alpha  - (1+\epsilon)\alpha^2)\psi_e - \alpha u_a \psi_e + u\cdot \nabla \psi_e - \Delta \psi_e + 2\alpha (\psi_e)_x]\\
			&\geq Ae^{-\alpha(s-R_\epsilon)} \left[\left(\frac{c_\epsilon}{2} - (1+\epsilon)\alpha\right)\alpha\psi_e - \alpha u_a \psi_e + u\cdot \nabla \psi_e\right]\\
			&\geq Ae^{-\alpha(s-R_\epsilon)} \left[\alpha\psi_e\left(\frac{c_\epsilon}{2} - \frac{c_\epsilon}{4}\right) - \alpha\psi_e \frac{c_\epsilon}{16} - \psi_e\frac{\alpha c_\epsilon}{16}\right]\\
			&= Ae^{-\alpha(s-R_\epsilon)}\psi_e \alpha \frac{c_\epsilon}{8} \geq 0.
	\end{split}
\end{equation*}
As in Lemma \ref{front_speed}, we can show that if $A$ is large enough, then $T^m \leq \gamma_A$ on $[R,a_n]\times\Omega_p$.  We can also argue that this inequality holds for all $A \geq A_0$ where $A_0 = \|\psi_e^{-1}\|_\infty$.  Hence $T_n^m \leq\gamma_{A_0}$ on $[R,a_n]\times\Omega_p$ when $n$ is sufficiently large, which finishes the proof.
\end{proof}

%This proof is similar to the argument appearing in ~\cite{CLR}.
%%%%
%%%%

\subsubsection*{The left and right limits of the temperature}

We will now show that $T^m$ converges to constants uniformly as $s\to \pm\infty$.

\begin{lem}\label{reg_end_limits}
There exists limits $\theta_{\pm}$ such that
\begin{equation}
	\lim_{s\to\pm\infty} \|T^m(s,x,z) - \theta_{\pm}\|_{L^\infty(\Omega_p)} = 0.
\end{equation}
Moreover, the limit on the right hand side is zero and the limit on the left hand side is positive.
\end{lem}
\begin{proof}
The limit on the right hand side follows from Corollary \ref{reg_right_limit}.  Hence we need only consider the limit on the left hand side.  The proof of this is handled exactly as in Lemma \ref{bdry_deriv_dies}.

The fact that this limit is positive follows from integrating equation \eqref{regularized_equation_infinite_domain} to obtain
\begin{equation}\label{reaction_rate}
c_\epsilon|\Omega_p| \theta_- = \int_{\R\times\Omega_p} f(T^m)dxdzds.
\end{equation}
As $c_\epsilon$
is positive,  positivity of $\theta_-$ is equivalent to the positivity of
the integral  on the right.  However, if  
the integral were zero, $T^m$ would solve an elliptic problem
and achieves a maximum, $\theta_0$, in the interior of the domain.
This contradicts the maximum principle.  Hence, it must be that
\[
	\int_{\R\times\Omega_p}f(T^m)dxdzds >0,
\]
and thus $\theta_-$ is positive. 
\end{proof}

\subsubsection*{Bounds in the stationary frame}

We will now change variables back to the stationary frame.  We will
obtain uniform bounds in $\epsilon$ in order to pass to the limit
$\epsilon \downarrow 0$.  Then we will only need to deal with the
$\delta$-dependence and the boundary conditions in the next section in
order to finish the proof of Theorem \ref{the_theorem}.

Define the function $T^\epsilon(t,x,z) = T^m(x - c_\epsilon t, x, z)$
on $\R \times \Omega$.  Make the same change of coordinates to define
the functions $\omega^\epsilon$, $u^\epsilon$, and $\Psi^\epsilon$.
We first get some local, uniform bounds on these
functions.  Here we follow the work of Berestycki and Hamel in
\cite{BH_Period}.

%%%%%
%%%%%
\begin{lem}\label{temp_grad_bd}
  For every compact set $K \subset \Omega$, there is a constant
  $C=C(\diam(K))>0$ that does not depend on $\epsilon$ such that
\begin{equation}\label{temp_spatial_sobolev_bd}
	\int_{\R\times K} |\nabla T^\epsilon|^2 dtdxdz
		\leq C.
\end{equation}
\end{lem}
\begin{proof}
  We begin by noting that $0\leq T^m \leq 1$ and that $T^m_x$,
  $T^m_s$, and $T^m_z$ tend to zero uniformly as $s$ tends to infinity
  since $T^m$ is bounded in $C^{1,\alpha}$ and $\nabla_{s,x,z} T^m$ is
  bounded in $L^2$ (by constants possibly depending on $\epsilon$).
  Hence we can multiply \eqref{regularized_equation_infinite_domain}
  by $T^m$ and integrate to obtain
\begin{equation*}
	c_\epsilon |\Omega_p|\frac{\theta_-^2}{2}
		+ \epsilon\int_{\R\times\Omega_p} (T^m_s)^2 dxdzds + \int_{\R\times\Omega_p}|\tn T^m|^2dxdzds
		= \int_{\R\times\Omega_p} T f(T)dxdzds.
\end{equation*}
Now, using equation \eqref{reaction_rate}, we get
\begin{equation*}
	\epsilon\int_{\R\times\Omega_p} (T^m_s)^2 dxdzds + \int_{\R\times\Omega_p}|\tn T^m|^2dxdzds
\le c_\epsilon |\Omega_p|\left(\theta_- - \frac{\theta_-^2}{2}\right).
\end{equation*}
Ignoring the first term and changing variables appropriately, we arrive at
\begin{equation*}
	\int_{\R\times\{(x,z)\in \Omega: x\in (k\ell, (k+1)\ell)\}}|\nabla T|^2dxdz
		\leq |\Omega_p|\left(\theta_- - \frac{\theta_-^2}{2}\right),
\end{equation*}
for any integer $k$.  To finish the proof, simply notice that given a compact set $K\subset \Omega$, we can find $k_1, \dots, k_n \in \Z$, where $n$ depends only on the diameter of $K$, such that
$$
K \subset \{(x,z): (x,z)\in \Omega, x \in [k_iL, (k_i+1)L] \mbox{ for some } i\}.
$$
\end{proof}

%%%%%
%%%%%
\begin{lem}\label{vort_grad_bd}
For every compact set $K \subset \Omega$, there is a constant
$C=C(\diam(K))>0$ that does not depend on $\epsilon$ such that
\begin{equation}\label{vort_sobolev_bd}
	\int_{\R\times K} |\nabla \omega^\epsilon|^2 dtdxdz + \int_{\R\times K} |\omega_t^\epsilon|^2 dtdxdz
		\leq C_K.
\end{equation}
\end{lem}
\begin{proof}
  Since $\omega^m \in H^1(\R\times\Omega_p)\cap
  C^{1,\alpha}(\R\times\Omega_p)$, then $\omega^m$ and its first
  derivatives tend to zero uniformly as $s$ tends to infinity.  Hence,
  multiplying equation \eqref{regularized_equation_infinite_domain} by
  $\omega^m$ and integrating over $\R\times\Omega_p$ yields
\begin{equation*}
	\epsilon\int_{\R\times \Omega_p} (\omega^m_s)^2 dsdxdz + \int_{\R\times\Omega_p} |\tn \omega^m|^2 dxdzds = \int_{\R\times \Omega_p} \omega^m(\hat{e}\cdot \tn T^m) dsdxdz
\end{equation*}
Using Poincar\'e's inequality and the non-negativity of $(\omega_s^m)^2$, we arrive at
\begin{equation*}
	\int_{\R\times\Omega_p} |\tn \omega^m|^2dxdzds
		\leq \left(\int_{\R\times \Omega_p} (\omega_z^m)^2 dsdxdz\right)^{1/2}\left(\int_{\R\times \Omega_p}(\tn T^m)^2 dsdxdz\right)^{1/2}
\end{equation*}
Changing variables and using Lemma \ref{temp_grad_bd}, finishes the
bound on the spatial gradient.

For the bound on the time derivative, first notice that $\omega_s^m$
is zero on one part of the boundary, and periodic on the other, and
tends to zero as $s$ tends to zero.  If we multiply
\eqref{regularized_equation_infinite_domain} by $\omega^m_s$, we
integrate to obtain
\begin{equation}\label{vort_grad_calc1}
	\begin{split}
		c_\epsilon\int_{\R\times \Omega_p} (\omega^m_s)^2dsdxdz
			&= -\int_{\R\times\Omega_p} \omega_s^m (\hat{e} \cdot \tn T^m) dsdxdz
				- \int_{\R\times\Omega_p} \omega_s^m L_\epsilon \omega^m dsdxdz\\
			&\leq \left(\int_{\R\times\Omega_p} (\omega_s^m)^2 dxdzds\right)^\frac{1}{2}\left(\int_{\R\times\Omega_p} |\tn T^m|^2 dxdzds\right)^\frac{1}{2}\\
			&\quad + \int_{\R\times\Omega_p} \left(\tn
                          \omega^m \cdot \tn \omega^m_s + 
(1+\epsilon)\omega_s \omega_{ss}\right)dxdsdz.
	\end{split}
\end{equation}
The last line comes from applying the Cauchy-Schwartz inequality to the first term and integrating the second term by parts.  Since the integrand of the last term is
\[
	\frac12\left(\partial_s |\tn \omega^m|^2 + \partial_s (\omega_s^m)^2\right),
\]
it  integrates to zero.  Hence, \eqref{vort_grad_calc1} becomes
\begin{equation*}
	\int_{\R\times \Omega_p} (c_\epsilon\omega^m_s)^2 dxdzds
		\leq \int_{\R\times\Omega_p} |\tn T^m|^2dxdzds.
\end{equation*}
Changing variables (which removes the factor of $c_\epsilon$) and
using Lemma \ref{temp_grad_bd} finishes the proof.
\end{proof}

Next we will prove bounds on the stream function and the fluid velocity using our estimates above.  However, since the equation for the stream function degenerates in the $t$ variable, we will first prove a general result concerning H\"older norms for degenerate families of elliptic equations.

\begin{lem}\label{2d_schauder_trick}
Let $U$ be a smooth open domain in $\R^n$, and let $\alpha \in (0,1)$.  Suppose that $\phi$ is a $C^2(\R\times U)$ solution to
\begin{equation*}
	\begin{split}
		\beta^2 \phi_{tt} + \Delta \phi &= g ~~~ \text{ on } \R\times U\\
		\phi &= 0~~~\text{ on } \R\times\partial U
	\end{split}
\end{equation*}
where $g$ is in bounded in $C^{0,\alpha}(\R\times U)$ and where $0< \beta < \beta_0$. Then there exists a constant $C$, which depends only on $\beta_0$, $\alpha$, and the smoothness of the boundary of $U$, such that
\[
\|\phi(t,\cdot)\|_{C^{2,\alpha}(U)}
	\leq C\left(\|\phi(t,\cdot)\|_{C^0(U)} + \| g \|_{C^{0,\alpha}(\R\times U)} \right),
\]
for every $t\in \R$.
\end{lem}
\begin{proof}
Define $\phi_\beta(t,x) = \phi(\beta t, x)$ and $g_\beta(t,x) = g(\beta t, x)$.  Then $\phi_\beta$ satisfies
\[
	\begin{split}
		\Delta_{t,x}\phi_\beta &= g_\beta ~~~\text{ on } \R\times U\\
		\phi_\beta &= 0~~~\text{ on } \R\times \partial U.
	\end{split}
\]
The usual Schauder estimates tell us that
\[
	\|\phi_\beta\|_{C^{2,\alpha}}
	\leq C\left(\|\phi\|_{C^0(U)} + \| g_\beta \|_{C^{0,\alpha}(\R\times U)} \right).
\]
Notice that
\[
\|g_\beta\|_{C^{0,\alpha}(\R\times U)} \leq (1+\beta^\alpha)\|g\|_{C^{0,\alpha}(\R\times U)}.
\]
Hence we get that
\[
	\|\phi(t,\cdot)\|_{C^{2,\alpha}(U)}
	\leq C\left(\|\phi\|_{C^0(\R\times U)} + (1+\beta_0)^\alpha\| f \|_{C^{0,\alpha}(\R\times U)} \right).
\]
\end{proof}

In order to use the result above in the equation for $\Psi$, we need to show that $\epsilon/c_\epsilon^2$ is bounded uniformly.  We do this below.  This will also be a crucial Lemma when we eventually show that the front speed, $c_\epsilon$, is bounded away from 0 uniformly in $\epsilon$.

%%%%%%%
% Lower bound on c_\epsilon^2/\epsilon
%%%%%%%
\begin{lem}\label{weak_lower_bound_front_speed}
There exists a universal constant $C>0$ such that
$$\frac{c_\epsilon^2}{\epsilon} \geq C$$
\end{lem}
%%%%%%%%%%%%
\begin{proof}
To begin, assume $n$ is large enough so that 
\[
\frac{c_\epsilon}{2} \leq c_{n, \epsilon} \leq \frac{3c_\epsilon}{2}.
\]
Recall from Lemma \ref{pos_front_speed} that for each $n$,
$$\left(\int f(T_n^{m,\epsilon}) dxdzds\right)\left(\int |(T_n^{m,\epsilon})_s|^2 dxdzds\right) \geq C.$$
Multiplying equation \eqref{eq_finiteprob1} by $(T^m_n - 1)$, integrating over $R_{a_n}$, and using that $(T^m_n -1)f(T^m_n) \leq 0$, we obtain
$$\epsilon \int |(T_n^{m,\epsilon})_s|^2dxdzds \leq \frac{c_\epsilon|\Omega_p|}{2} - (1+\epsilon)\int_{\Omega_p} (T^{m,\epsilon}_n)_s(a_n,x,z)dxdz.$$
Since, by Lemma \ref{bdry_deriv_dies}, the second term here tends to zero uniformly, then by choosing $n$ large enough, we can combine the above equations to get
$$\frac{c_\epsilon}{\epsilon} \int f(T_n^{m,\epsilon})dxdzds \geq C.$$
To get a bound on the reaction term, we simply take  the equation for $T_n^{m,\epsilon}$ in \eqref{eq_finiteprob1} and integrate it to get
$$c_{\epsilon,n}|\Omega_p| + (1+\epsilon)\int_{\Omega_p} \left((T_n^{m,\epsilon})_s(-a_n,x,z) - (T_n^{m,\epsilon})_s(a_n,x,z)\right)dxdz = \int f(T_n^{m,\epsilon})dxdzds.$$
The first integral has two terms in it.  One is non-positive and the other tends uniformly to zero as $n$ tends to zero by Lemma \ref{bdry_deriv_dies}.  Hence we combine this with the above equation to obtain the desired inequality.
\end{proof}

Because of the degeneracy in the equation for $\Psi$, it will be convenient for us to use partial Sobolev norms in the following lemma.  We define these below.

\begin{defn}
  If $U\subset \R\times\R^n$ is a simply connected domain with a
  Lipshitz boundary, then a function $f$ is in the space $X^{i,j}(U)
  \subset L^2(U)$ if for every multi-index $\beta$, $f$ satisfies
\[
\int_U |\partial_{t}^{\beta_0}\partial_{x_1}^{\beta_1}\cdots\partial_{x_{n}}^{\beta_{n}} f(x)|^2 dx < \infty,
\]
where $\beta_0 \leq i$ and $|\beta| \leq j$.  We
endow the space with the norm
\[
\|f\|_{X^{i,j}(U)}^2 = \sum_{\substack{|\beta|\leq j,\\\beta_1 \leq i}}\int_U |\partial_{x_1}^{\beta_1}\cdots\partial_{x_{n+1}}^{\beta_{n+1}} f(x)|^2 dx.
\]
\end{defn}

%%%%%%%%
% BOUNDS ON L2 NORM AND HOLDER NORM OF U
%%%%%%%%
\begin{lem}\label{flow_bounds}
  For every compact set $K \subset \Omega$, there is a constant
  $C=C(\diam(K))>0$ that does not depend on $\epsilon$ such that
\begin{equation}\label{flow_sobolev_bd}
	\int_{\R\times K} |\nabla u^\epsilon|^2 dtdxdz + \int_{\R\times K} |u_t^\epsilon|^2 dtdxdz + \int_{\R\times K} |\nabla \Psi^\epsilon|^2 dtdxdz + \int_{\R\times K} |\Psi_t^\epsilon|^2 dtdxdz
		\leq C_K.
\end{equation}
In addition, there is a constant $C_\delta>0$, depending only on
$\delta$, such that the following bounds hold on $\R\times\Omega$:
$$\|u\|_{L^\infty} \leq C_\delta.$$
\end{lem}
\begin{proof}
  We first prove the Sobolev bounds.  Multiply equation
  \eqref{regularized_equation_infinite_domain} by $\Psi^m$ and
  integrate by parts.  The boundary terms vanish because $\Psi^m$ and
  its derivatives vanish at infinity.  We obtain
$$
\int_{\R\times \Omega_p} |\tn \Psi^m|^2dsdxdz
	\leq C \|\Psi^m\|_{L^2(\R\times\Omega_p)} \|\tilde\omega^m\|_{L^2(\R\times\Omega_p)}.
$$
Then we change variables and use the Poincar\'e inequality to finish
the bound on $\nabla\Psi^\epsilon$.
We get a similar estimate for $\nabla \Psi_t^\epsilon$ by differentiating
\eqref{regularized_equation_infinite_domain} in $s$ and then arguing
in the same manner as above, which gives us the bound on
$u_t^\epsilon$. 
The bound on $\Psi_t^\epsilon$ follows
then from the Poincar\'e inequality again.  Finally, we need to get the bound
on
$\nabla u^\epsilon$.  Note that the usual elliptic estimates, combined with our work above, give us that
\[
	\|\Psi\|_{H^3(\R\times K)} \leq C\frac{c_\epsilon^2}{\epsilon} \| \tilde\omega^m\|_{H^1(\R\times K)}.
\]
The constant, $C$, in the equation above is universal.  In particular, we have that $(\epsilon/c_\epsilon^2) \Psi_{tt}$ is bounded in $L^2$.  Since $\Psi$ satisfies
\[
	\Delta \Psi = \tilde\omega - \frac{\epsilon}{c_\epsilon^2} \Psi_{tt},
\]
we obtain, via the usual elliptic estimates in the spatial coordinates,
\[
\|\Psi\|_{X^{0,3}(\R\times K)}
	\leq C \left(\|\tilde\omega\|_{H^1(\R\times K)} + \left\|\frac{\epsilon}{c_\epsilon^2} \Psi_{tt} \right\|_{H^1(\R\times K)} \right)
	\leq C \|\tilde{\omega} \|_{H^1(\R\times K)}.
\]
The last bound comes from differentiating \eqref{regularized_equation_infinite_domain} in $s$, multiplying it by $\Psi_s$, integrating by parts, and using the Cauchy-Schwarz and Poincar\'e inequalities to obtain
\[
	\int_{\R\times \Omega_p} |\tn \Psi_s|^2 dxdzds \leq C \int_{\R\times \Omega_p} |\tilde\omega_s|^2 dxdzds.
\]
Changing variables and arguing as we did above gives us
\[
\|\Psi_t\|_{X^{0,2}(\R\times K)}
	\leq C \left(\|\tilde\omega_t\|_{L^2(\R\times K)} + \left\|\frac{\epsilon}{c_\epsilon^2} \Psi_{ttt} \right\|_{L^2(\R\times K)} \right)
	\leq C \|\tilde{\omega} \|_{H^1(\R\times K)}
\]
Combining this with our estimates above yields
\[
\|\Psi\|_{X^{1,3}(\R\times K)}
	\leq C \|\omega\|_{H^1(\R\times K)},
\]
where $C$ is a universal constant.

Notice that our work above gave us an $L^\infty$ bound on $\Psi$ by Theorem 2.2.6 from~\cite{KrylovSobolev}.  Moreover, $\Psi$ satisfies the equation
\[
	\frac{\epsilon}{c_\epsilon^2}\Psi_{tt} + \Delta \Psi = \tilde\omega.
\]
Hence, applying Lemma~\ref{2d_schauder_trick} and noting Lemma \ref{weak_lower_bound_front_speed} finishes the proof for us.

\end{proof}

Now that we have a bound for $u$ which is
independent of $\epsilon$, we may get an
upper bound on $c_{\epsilon,\delta}$ which is independent of
$\epsilon$.

%%%%%%%%
%%%%%%%%
% 	UPPER BOUND ON FRONT SPEED
%%%%%%%%

\begin{cor}\label{front_speed_upper_bound_independent_of_epsilon}
There exists a universal constant $C>0$ such that
\[
c_{\epsilon,\delta} \leq C(1 + \|u^{\epsilon,\delta}\|_{L^\infty}).
\]
\end{cor}
\begin{proof}
We can assume without loss of generality that $c_\epsilon \geq 24$ for
otherwise there is nothing to prove. Choose $n$ large enough such that 
$c_{n,\epsilon} \geq c_\epsilon/2$.  After possibly passing to a
subsequence, we may assume that
$u_n^{m,\epsilon}$ satisfies 
\[
\|u_n^{m,\epsilon}\|_{L^\infty([-n,n]\times\Omega_p)} \leq 
2 \|u^{m,\epsilon} \|_{L^\infty}.
\]
Let $\gamma(s,x,z) = \psi_e(x,z) e^{- (s+n)}$ where $\psi_e$ to be the
principal 
eigenfunction of the operator 
\[
-\Delta + 2\partial_x,
\]
with the boundary conditions $\eta_1\psi_e - \eta \cdot\nabla \psi_e = 0$ on $B$ and periodic on the boundary $P$.  Letting $A_0 = \|\psi_e^{-1}\|_{\infty}$ and arguing as before, we have that $\gamma_{A_0}(s) \geq T^m_n$ on $[-n,n]\times \Omega_p$ as long as the following conditions hold
\begin{equation}\label{supersolution_conditions}
	\begin{split}
		c_n &\geq M + 2 + \|u^m_n\|_{L^\infty([-n,n]\times\Omega_p)}\left( 1 + \left\| \frac{\nabla\psi_e}{\psi_e}\right\|_{L^\infty{\Omega_p}}\right)\\
		A_0e^{-2 n} &\geq T_n^m(n,x,z).
	\end{split}
\end{equation}
Notice that the second equation holds for large enough $n$, by simply taking $\alpha = 3$ in Corollary \ref{reg_right_limit}.  Hence, if the first equation holds then we obtain
$$A_0e^{-n}\geq \gamma(0) \geq \max T_n^m(0, x,z) = \theta_0.$$
This is clearly not true for $n$ large.  Hence it must be that the first inequality in \eqref{supersolution_conditions} is false.  This leads to the inequality
$$
c_\epsilon \leq 2c_n \leq 2\left(M + 2 + 2\|u^m\|_{L^\infty(\R\times\Omega_p)}\left( 1 + \left\| \frac{\nabla\psi_e}{\psi_e}\right\|_{L^\infty{\Omega_p}}\right)\right).
$$
This finishes the proof.
\end{proof}

Finally, we get an $L^2$ bound on the time derivative of the temperature.

%%%%%%%%%
%%%%%%%%%
% BOUND ON GRADIENT OF T
%%%%%%%%%
\begin{lem}\label{temperature_time_derivative}
For every compact set $K \subset \Omega$, there is a constant $C_\delta=C(\delta,\diam(K))>0$, which depends only on $\delta$ and the diameter of $K$, such that 
\begin{equation}\label{sobolev_temp}
	\int_{\R \times K} (T^\epsilon_t)^2 \leq C_\delta.
\end{equation}
\end{lem}
\begin{proof}
Multiply equation \eqref{regularized_equation_infinite_domain} by $T_s^m$, integrate over $[-A,A]\times\Omega_p$, and integrate by parts to obtain
\begin{equation}\label{sobolev_temp1}
	\begin{split}
		-c_\epsilon \int_{[-A,A]\times\Omega_p} &(T_s^m)^2 dsdxdz
			- (1+\epsilon) \int_{[-A,A]\times\Omega_p} T_{ss}^mT_s^mdxdzds
			\\&+ \int_{[-A,A]\times\Omega_p} \nabla_{x,y} T^m \cdot \nabla_{x,y} T_s^m dxdzds
			+ \frac{1 + \epsilon}{2}\int_{\Omega_p} [(T_s^m)^2]_{-A}^A dxdz
			\\&+ \int_{[-A,A]\times\Omega_p} (u\cdot\tn T^m) T^m_s dxdzds
			= \int_{[-A,A]\times\Omega_p} T_s^m f(T)dxdzds.
	\end{split}
\end{equation}
First notice that, as usual, the boundary terms will tend to zero as $A$ tends to infinity.  The second and third terms will tend to zero as we take $A$ to infinity since we can write the integrands as $\partial_s (T_s^m)^2$ and $\partial_s |\tn T^m|^2$, respectively.  Hence, there is a function $\eta(A)$ which tends to zero as $A$ tends to infinity which bounds the second, third, and fourth term.  We write the last term as
$$\int_{[-A,A]\times\Omega_p} T_s^m f(T^m)dxdzds
	= \int_{\Omega_p} [F(T^m)]_{-A}^A dxdz$$
where
\[
F(t) = \int_0^t f(\tau)d\tau.
\]
Let $C$ be a uniform bound on $f(t)$.  Then, combining this with \eqref{sobolev_temp1}, we arrive at
\begin{equation}
	\begin{split}
		c_\epsilon \int_{[-A,A]\times\Omega_p} (T_s^m)^2 dsdxdz
			&\leq \eta(A)
			- \int_{[-A,A]\times\Omega_p} (u\cdot\tn T^m) T^m_s dxdzds
			- \int_{\Omega_p} [F(T^m)]_{-A}^Adxdz
			\\&\leq \eta(A)
			- \int_{[-A,A]\times\Omega_p} (u\cdot\tn T^m)T^m_s dxdzds
			+ C
			\\&\leq \eta(A)
			+ \int_{[-A,A]\times\Omega_p} \frac{\|u^m\|_{L^\infty}}{2} \left(\alpha |\tn T^m|^2 + \frac{(T^m_s)^2}{\alpha} \right) dxdzds
			+ C.
	\end{split}
\end{equation}
Here, we used the boundedness of $F$ and the bound on $u^m$ that is uniform in $\epsilon$.  The following holds for any choice of $\alpha$ which is positive.  If $u \equiv 0$ then the desired inequality holds by simply skipping the last step in the calculations above.  Otherwise, let $\alpha = \|u^m\|_{L^\infty}/c_\epsilon$.  Then, taking $A \to \infty$, we arrive at
\begin{equation*}
\frac{c_\epsilon}{2}\int_{\R\times\Omega_p} (T_s^m)^2dxdzds
	\leq C 
	+ \int_{\R\times\Omega_p} \frac{\|u\|_{\infty}^2}{2c_\epsilon} |\tn T^m|^2dxdzds.
\end{equation*}
From Lemma \ref{temp_grad_bd}, we get that
\begin{equation*}
	\int_{\R\times\Omega_p} (c_\epsilon T^m_s)^2 dxdzds
		\leq C_\delta c_\epsilon.
\end{equation*}
Changing variables and arguing as in Lemma \ref{temp_grad_bd}, finishes the proof.
\end{proof}

In order to take the limit as $\epsilon$ tends to zero we need a lower
bound on the front speed $c_\epsilon$.  We obtain that here, in the
next lemmas.  Recall that we have already eliminated the troublesome case
where $c_\epsilon$ tends to zero faster than $\sqrt{\epsilon}$ in Lemma \ref{weak_lower_bound_front_speed}.

\subsubsection*{A lower bound on the front speed}

Now we will complete the proof that $c_\epsilon$ is uniformly bounded away from zero as $\epsilon$ tends to zero.

%%%%%%
%%%%%%
%%%%%%
\begin{prop}\label{pos_front_speed_unregularized}
There exists a constant $C_\delta>0$ that depends only on $\delta>0$
so  that $c_\epsilon>C_\delta$ for all $\eps$ sufficiently small. 
%
%$$0 < \liminf_{\epsilon\to0} c_\epsilon.$$
%
\end{prop}
\begin{proof}
We will prove this by contradiction.  Assume that there is a sequence
$\epsilon_n \to 0$ such that $c_{\epsilon_n} \to 0$ as well.  For
simplicity, we drop the $n$ from the 
notation.  Notice that $\epsilon/c_\epsilon^2$ is uniformly bounded so it must converge, along a subsequence if necessary, to a constant $\kappa \geq 0$.  In order to come to a contradiction, we build new solutions to our equations as follows.  Let
\begin{equation}\label{front_speed_positive_recenter}
s_n = \sup\left\{s: \frac{1+\theta_0}{2}
	= \min_{r\leq s} \frac{1}{|\{(x,z)\in\Omega_p: x \in[0, c_\epsilon]\}|}\int_0^{c_\epsilon}\int_{\Omega_p(x)} T_n^m(r, x, z) dx dz\right\},
\end{equation}
where $\Omega_p(x) = \{z\in \R: (x,z) \in \Omega_p\}$, and define re-centered functions 
\[
\Phi_n^m(s,x,z) = T_n^m(s + s_n, x, z), ~~U_n(s,x,z) = u(s+s_n, x, z).
\]
These functions satisfy the same equations and bounds as before on 
$[-a_n - s_n, a_n - s_n]\times \Omega_p$ so we can take the limit as
$n$ 
tends to infinity.  We let 
\[
b_\epsilon = \lim_{n\to+\infty} (-a_n - s_n).
\] 
 
%  Either $b_\epsilon$ is $-\infty$ for a sequence of $\epsilon$ tending to zero, it is bounded uniformly in $\epsilon$, or there is a sequence of $\epsilon$ along which it is finite for each $\epsilon$ but such that it tends towards $-\infty$.

For each $\epsilon$, passing to the limit $n\to+\infty$ we also get solutions $\Phi_\epsilon^m, U_\epsilon^m$ to equation \eqref{regularized_equation_infinite_domain} on $(-b_\epsilon, \infty)\times \Omega_p$.  Changing variables to the stationary frame, we arrive at solutions $\Phi^\epsilon$, $U^\epsilon$, which solve the equation
\begin{equation*}
	\Phi^\epsilon_t + U^\epsilon \cdot \nabla \Phi^\epsilon = \Delta \Phi^\epsilon + \frac{\epsilon}{c_\epsilon^2} \Phi^\epsilon_{tt} + f(\Phi^\epsilon).
\end{equation*}
Let $\Omega(x) = \{z: (x,z)\in \Omega\}$.  Notice that our choice of $s_n$ in \eqref{front_speed_positive_recenter} gives us that
\begin{equation}\label{crosssection1_eps}
\frac{1}{\int_0^1|\Omega(c_\epsilon t)|dt}\int_0^1 \int_{\Omega(c_\epsilon t)} \Phi^{\epsilon}(t,c_\epsilon t,z) dzdt
	= \frac{1+\theta_0}{2}
\end{equation}
and that
\begin{equation}\label{crosssection2_eps}
\frac{1}{\int_0^1|\Omega(x_0 + c_\epsilon t)|dt}\int_0^1 \int_{\Omega(x_0 + c_\epsilon t)} \Phi^{\epsilon}(t,x_0 + c_\epsilon t,z) dzdt
	\geq \frac{1+\theta_0}{2}
\end{equation}
as long as $b_\epsilon \leq x_0 \leq 0$.  Since we have a uniform $H_{loc}^1$ bound on both functions, we can take $\epsilon$ to zero (along a subsequence if necessary) to get functions $\Phi$, $U$ which weakly solve the equation
\begin{equation}\label{apr902}
	\Phi_t + U\cdot \nabla \Phi = \Delta \Phi + \kappa \Phi_{tt}+ f(\Phi).
\end{equation}
Also, taking a subsequence if necessary, $b_\epsilon$ converges
as $\epsilon\to 0$ to either a finite number or~$-\infty$, and we set 
\[
b_0 = \lim_{\epsilon\to 0} b_\epsilon.
\]
Equation (\ref{apr902}) is posed on
the set $\Omega^{b_0} = \{(t,x,z)\in \R\times \Omega: x \geq b_0\}$.  
Also, by using the trace theorem, we get, from \eqref{crosssection1_eps} and \eqref{crosssection2_eps} that 
\begin{equation}\label{crosssection1}
\frac{1}{|\Omega(0)|}\int_0^1 \int_{\Omega(0)} \Phi^{\epsilon}(t,0,z) dzdt
	= \frac{1+\theta_0}{2}
\end{equation}
and that
\begin{equation}\label{crosssection2}
\frac{1}{|\Omega(x_0)|}\int_0^1 \int_{\Omega(x_0)} \Phi^{\epsilon}(t,x_0,z) dzdt
	\geq \frac{1+\theta_0}{2}
\end{equation}
as long as $b_0 \leq x_0 \leq 0$.  The last ingredient that we need is a global $L^2$ bound on the derivatives of some functions.  Looking at how we obtained the $L^2$ gradient bounds earlier in this section, notice that if we had changed variables in a different way, we have, for each $\epsilon$, bounds of the form
\begin{equation}\label{gradient_bounds_eps}
	\begin{split}
		\int_0^{\ell/c_\epsilon}\int_{\Omega^{b_\epsilon + c_\epsilon t}} |\Phi_t^\epsilon|^2 + |\nabla \Phi^\epsilon|^2 dxdzdt
			&= \int_{-\ell/c_\epsilon}^0\int_{\Omega^{b_\epsilon + c_\epsilon t}} |\Phi_t^\epsilon|^2 + |\nabla \Phi^\epsilon|^2 dxdzdt
			\leq C,\\
		\int_0^{\ell/c_\epsilon}\int_{\Omega^{b_\epsilon + c_\epsilon t}} |U_t^\epsilon|^2 + |\nabla U^\epsilon|^2 dxdzdt
			&= \int_{-\ell/c_\epsilon}^0\int_{\Omega^{b_\epsilon + c_\epsilon t}} |U_t^\epsilon|^2 + |\nabla U^\epsilon|^2 dxdzdt
			\leq C
	\end{split}
\end{equation}
Similarly, from the identity $\int_{\R\times\Omega_p} f(T^{m,\epsilon})dxdzds \leq Cc_\epsilon$, we arrive at
\begin{equation}\label{reaction_bound_eps}
	\int_0^{\ell/c_\epsilon}\int_{\Omega^{b_\epsilon + c_\epsilon t}} f(\Phi^\epsilon) dxdzdt
		= \int_{-\ell/c_\epsilon}^0\int_{\Omega^{b_\epsilon + c_\epsilon t}} |f(\Phi^\epsilon) dxdzdt
		\leq C.
\end{equation}
Now taking $\epsilon$ to zero in equations \eqref{gradient_bounds_eps} and \eqref{reaction_bound_eps} gives us the following bounds
\begin{equation}
	\int_{\R}\int_{\Omega^{b_0}} \left[f(\Phi) + |\nabla \Phi|^2 + |\Phi_t|^2 + |\nabla U|^2 + |U_t|^2 \right] dxdzdt \leq C.
\end{equation}

By parabolic or elliptic regularity (depending on whether $\kappa = 0$ or $\kappa >0$, respectively), it follows that $\Phi$ has a uniform $C^{1,\alpha}$ bound in both space and time, see e.g. ~\cite{GT, KrylovHolder, KrylovSobolev}.  Let $0< \nu < \frac{1-\theta_0}{2}$ and choose $\mu$ small enough that $f(T) < \mu$ implies that $T \leq \theta_0 + \nu$ or $T \geq 1 - \nu$.  Then take $R$ large enough such that if $|t| + |x| \geq R$, we have that $f(\Phi(t,x,z)) < \mu$.  This exists since $f(\Phi)$ satisfies a global $L^2$ bound and a global $C^{1,\alpha}$ bound.  Since $\{|t| + |x|\geq R\}$ is connected and since $f(\Phi)$ is continuous then either $\Phi(t,x,z) \geq 1 - \nu$ on $\{|t|+|x|\geq R\}$ or $\Phi(t,x,z) \leq \theta_0 + \nu$ on $\{|t|+|x|\geq R\}$.

We claim that the first possibility holds.  To see this, we look at the two different cases separately: either $b_0 = -\infty$ or $0 \leq -b_0 < \infty$.  In the first case, we use \eqref{crosssection2}.  In the second case we use the fact that $T(t, b_0, z) = 1$ for all $t\in\R$ and $z \in \Omega(b_0)$.

To review, we have that if $|t| + |x| \geq R$ then  
$T(t,x,z) \geq 1 - \nu$.  In addition,   there is some $t_0 \in (0,1)$ and $z_0 \in \Omega(0)$ such that $T(t_0, 0, z_0) = \frac{1+\theta_0}{2}$.  We claim that this leads to a contradiction.  Look at the domain $[-R,R]\times \{(x,z)\in \Omega^{b_0}: |x| \leq R\}$.  Then the Neumann boundary conditions on $T$ and the strong maximum principle (either for parabolic or elliptic equation, depending on whether $\kappa = 0$ or $\kappa >0$, respectively) tells us that $T(t,x,z) \geq 1-\nu$ for every $(t,x,z)$ in this domain.  This is a contradiction since $T(t_0, 0, z_0) =\frac{1+\theta_0}{2}$.
\end{proof}

A direct result of this lemma is that we can estimate $\tilde\omega$ in any Sobolev norm by the $L^2$ norm of $\omega$.  This allows us to conclude that $\Psi$ is smooth and hence, that its derivatives decay at infinity.  We will need this in the next section in order to justify integration by parts when we obtain new estimates on our functions and to get a bound on the $L^\infty$ norm of the $\nabla T$.

\begin{cor}\label{fluid_smoothness}
For every compact subset, $K$, of $\Omega$, the stream function, $\Psi$,is bounded in $H^k(\R\times K)$ for every $k$ by a constant depending only on $\delta$ and the diameter of $K$.
\end{cor}
\begin{proof}
First notice that for any natural number $k$,
\[
	\|\tilde\omega\|_{H^k(\R\times K)} \leq C \|\omega\|_{L^2(\R\times K)}.
\]
The constant above depends only on $\delta$.  Then arguing as in Lemma \ref{flow_bounds} we see that
\[
\int_{\R\times K} |\nabla \partial_t^j \Psi|^2 dxdzdt
	\leq C\int_{\R\times K} |\partial_t^j \tilde\omega|^2 dxdzdt
\]
and
\[
	\|\Psi\|_{H^{k+2}(\R\times K)}
	\leq C \frac{c_\epsilon^2}{\epsilon} \left( \|\psi\|_{L^2(\R\times K)} + \|\tilde\omega\|_{H^k(\R\times K)}\right)
	\leq C \frac{c_\epsilon^2}{\epsilon} \|\omega\|_{L^2(\R\times K)}.
\]
Fix $j$ and $k$ and notice that $\partial_t^j \Psi$ satisfies
\[
\begin{split}
	\Delta \partial_t^j \Psi &= \partial_t^j\tilde\omega - \frac{\epsilon}{c_\epsilon^2} \partial_t^{j+2} \Psi ~~~\text{ on } \R\times \Omega\\
	\partial_t^j \Psi &= 0 ~~~\text{ on } \R\times \partial\Omega	
\end{split}
\]
Hence, the usual elliptic estimates tell us that
\[
\begin{split}
	\|\partial_t^j \Psi\|_{X^{0,k}}
	&\leq C\left( \|\partial_t^j \psi\|_{L^2(\R\times K)} + \|\partial_t^j\tilde\omega -   \frac{c_\epsilon^2}{\epsilon} \partial_t^{j+2}\Psi\|_{H^k(\R\times K)}\right)\\
	&\leq C\left(\|\tilde\omega\|_{H^{j+k}(\R\times K)} +   \frac{\epsilon}{c_\epsilon^2} \|\Psi\|_{H^{j+k+2}(\R\times K)}\right)\\
	&\leq C \|\omega\|_{L^2(\R\times K)}.
\end{split}
\]
The estimate of $\Psi$ is finished by noting that $\|\omega\|_{L^2(\R\times K)}$ is uniformly bounded in $\epsilon$.
\end{proof}

%%%%%%%%%%%%%%%%%%%%%%%%%%%%%%%%%%%%%%%%%%%%%%%
%%%%%%%%%%%%%%%%%%%%%%%%%%%%%%%%%%%%%%%%%%%%%%%
%%%%%%%%%%%%%%%%%%%%%%%%%%%%%%%%%%%%%%%%%%%%%%%
%%%%%%									  %%%%%
%%%%%%									  %%%%%
%%%%%% Solution of the Unregularized Eqn  %%%%%
%%%%%%									  %%%%%
%%%%%%									  %%%%%
%%%%%%%%%%%%%%%%%%%%%%%%%%%%%%%%%%%%%%%%%%%%%%%
%%%%%%%%%%%%%%%%%%%%%%%%%%%%%%%%%%%%%%%%%%%%%%%
%%%%%%%%%%%%%%%%%%%%%%%%%%%%%%%%%%%%%%%%%%%%%%%

\section{Solution of the Unregularized Equation}\label{solutions_of_the_unregularized_equation}

%%%%  NEEDS TO BE REWORDED!!!!

The results of the previous section allow us to take the limit as
$\epsilon$ tends to zero in all the relevant topologies.  Thus, we
arrive at $T_\delta \in H^1_{loc}(\R\times \Omega)$, $u_\delta\in
C^{0,\alpha}(\R\times \Omega) \cap H^1_{loc}(\R\times\Omega)$,
$\omega_\delta \in H^1_{loc}(\R\times\Omega)$, and $\Psi_\delta \in
H_{loc}^1(\R\times\Omega)$.  These satisfy the system
\begin{equation}\label{stationary_frame_eqn}
	\begin{split}
		(T_\delta)_t + u_\delta\cdot \nabla T_\delta - \Delta T_\delta &= f(T_\delta) \\
		(\omega_\delta)_t - \Delta \omega_\delta &= \hat{e}\cdot \nabla T_\delta \\
		\Delta \Psi_\delta &= \tilde\omega_\delta\\
		u_\delta &= \nabla^\perp \Psi_\delta
	\end{split}
\end{equation}
on $\R \times \Omega$ with boundary conditions
\begin{equation}\label{sationary_frame_eqn}
	\begin{split}
		\frac{\partial T_\delta}{\partial \eta} & = 0\\
		\Psi_\delta & = 0\\
		\omega_\delta & = 0 \\
		u_\delta \cdot \eta &= 0
	\end{split}
\end{equation}
on $\partial \Omega$.

The last remaining step is to pass to the limit $\delta\to 0$.
Before we do that, we will need to check a few properties for each fixed $\delta$. First, we
will check that, for each $\delta$, the functions we have obtained at
this point are pulsating fronts as in equation
\eqref{pulsating_front}.  This is necessary to eventually show that the
functions we obtain as we take $\delta$ to zero are pulsating fronts.  Then we will show that $T_\delta^m$ can be bounded by an exponential as in Corollary \ref{reg_right_limit}.  Finally, we will discuss some uniform in $\delta$ estimates on our functions.  This will allow us to take the limit as $\delta$ tends to zero.  Finally, we will finish the proof of Theorem \ref{the_theorem}.

Here we will show that the $T_\delta$ satisfies the normalization condition, as in equation \eqref{normalization}.  We will need this later, to show that this condition holds in the limit as $\delta$ tends to zero.

\begin{lem}\label{pinned_down_at_middle}
For $T_\delta$, as constructed above,
$$\max\{T_\delta(t,x,z) : (t,x,z) \in \R\times\Omega: x-ct \leq 0\} = \theta_0.$$
\end{lem}
\begin{proof}
The main problem here is that the convergence of $T^\epsilon_\delta$ to $T_\delta$ is in $H^1$, which is not strong enough to guarantee that this condition carries over.  Certainly the inequality $T_\delta(t,x,z) \leq \theta_0$ holds as a result of our construction of $T$, so what remains is to show that $T_\delta$ takes the value $\theta_0$ somewhere on the set $\{(t,x,z)\in\R\times\Omega: x - ct \leq 0\}.$

To this end, we use the bounds in Lemma \ref{fluid_smoothness} and the results of Berestycki and Hamel in~\cite{BH_Gradient} to get a uniform gradient bound on $T_\delta^\epsilon$.  This paper implies that since $u$ is bounded in $C^1$, there is a constant $C>0$ which is independent of $\epsilon$, though which does depend on $\delta$, such that $\|\nabla_{x,z}T^\epsilon_\delta\|_{\infty} < C$.  This, along with the trace theorem, gives us the existence of a point such that $T_\delta^\epsilon(t, x, z) = \theta_0$ and such that $x - ct \leq 0$.
\end{proof}

Now we have to worry about the limits as $t\to \pm\infty$.  Here it
will be convenient to use the functions in the moving frame.  This is
justified since the front speed $c_\delta$ is positive, and the norms in the moving and stationary frames are equivalent, up to a factor of $c_\delta$.

We first show that on the right, $T^m$ tends to zero.  In order to prove this, we wish to show that $T^m$ is bounded by an exponential function on $[R,\infty) \times \Omega_p$ for some $R$ as we did in Lemma \ref{flow_tends_to_zero} and Corollary \ref{reg_right_limit}.  The proof is the same in spirit but altered slightly since we do not begin with the finite domain problem here.

\begin{lem}\label{right_uniform_limit_zero_moving}
Let
$$R_\epsilon = \sup\{ r \in [0,\infty): c_\delta/10 \geq \max_{\substack{(x,z)\in \Omega_p,\\ s \geq r}} u^{m,\epsilon}(s, x, z)\}.$$
For $\epsilon$ sufficiently small, there is a constant $C_\delta>0$, which does not depend on $\epsilon$, such that 
\[
	R_\epsilon < C_\delta.
\]
\end{lem}
\begin{proof}
Suppose that there is a subsequence $\epsilon_n \downarrow 0$ such that $R_{\epsilon_n} \to \infty$.  For ease of notation, we drop the ``$n$'' notation.  In this case, we recenter our equations so that
\[
\begin{split}
\Phi^{m,\epsilon(s, x, z)} = T^{m,\epsilon}(s + R_\epsilon, x, z),& ~~U^{m,\epsilon} (s,x,z) = u^{m, \epsilon}(s+R_\epsilon, x, z),\\
W^{m,\epsilon}(s,x,z) = \omega^{m,\epsilon}(s+R_\epsilon,x,z),& ~~S^{m,\epsilon}(s,x,z) = \Psi^{m,\epsilon}(s+R_\epsilon, x,z).
\end{split}
\]
These functions satisfy the same bounds as before so we can take limits as $\epsilon$ tends to zero, to obtain function $\Phi^m$, $U^m$, $W^m$, and $S^m$.  Since $R_\epsilon\to \infty$, we get that the $\Phi^m \leq \theta_0$.  Moreover, on $[0,\infty)\times \Omega_p$, $\Phi^m$ is bounded by the exponential defined in Corollary \ref{reg_right_limit}.  Hence the limit as $s$ tends to infinity is zero.  On the other hand, similar arguments as in Lemma \ref{bdry_deriv_dies} give us that there is a limit $\theta_-\in [0,\theta_0]$ such that
$$\theta_- = \lim_{s\to -\infty} \Phi^m(s,x,z),$$
where the limit is uniform in $\Omega_p$.  Hence integrating
$$-c_\delta\Phi^m_s + U \cdot \tn \Phi^m - L \Phi^m = 0$$
gives us that 
$$c_\delta\theta_- = 0.$$
Since $c_\delta$ is positive by Lemma \ref{pos_front_speed_unregularized}, we get that $\theta_-$ is zero.  In the stationary frame, where 
\[
\Phi(t,x,z) = \Phi^m(x - c_\delta t, x, z), ~~U(t,x,z) = U^m(x-c_\delta t, x,z),
\]
notice that $\Phi$ satisfies
\[
	\Phi_t + U\cdot \nabla \Phi = \Delta \Phi.
\]
Hence, the Hopf maximum principle along with the fact that $\theta_- = 0$ implies that $\Phi \equiv 0$.  This implies that $\Phi^m \equiv 0$.

Hence we have that $W^m$ satisfies
\[
	-c_\delta W^m_s - L W^m = 0,
\]
with Dirichlet boundary conditions on one boundary of $\Omega_p$ and
periodic boundary conditions on the other boundary.  Integrating this
over $\R\times \Omega_p$ and using a Poincar\'e inequality gives us
that $W^m \equiv 0$.  Therefore, $S^m$ satisfies
\[
	-L S^m = 0,
\]
with Dirichlet boundary conditions on one boundary of $\Omega_p$ and
periodic boundary conditions on the other boundary, whence $S \equiv 0$.  This finally implies that $U$ must be zero because of its relationship to $S$.
This contradicts the fact that $\max_{\Omega_p} U^m(0,x,z) = c_\delta/10>0$.  Hence it must be that $R_\epsilon$ is bounded.
\end{proof}

This result allows us to bound $T^m_\delta$ above by an exponential.  Eventually, we will show that the parameters in this exponential are bounded.  Hence, when we eventually take the limit $\delta\to0$, this bound will be preserved.

\begin{cor}\label{right_uniform_limit_zero}
For every $\delta$ and every $0<\alpha \leq c_\delta/8$, there are constants $C_\alpha, R_\delta>0$, which depend only on $\alpha$ and $\delta$, respectively, such that
$$
	T^m_\delta(s,x,z) \leq C_\alpha e^{- \alpha(s - R_\delta)}. 
$$
\end{cor}
\begin{proof}
This is simply a result of combining the conclusions of Corollary \ref{reg_right_limit} and Lemma \ref{right_uniform_limit_zero_moving}.
\end{proof}

The last property that we need to check is that the functions satisfy the first condition of equation \eqref{pulsating_front}.  Later, we will use this to show that this still holds in the limit $\delta \to 0$.

\begin{lem}\label{periodicity}
The functions $T_\delta$ and $u_\delta$ satisfy the condition \eqref{pulsating_front}.
\end{lem}
\begin{proof}
The argument for this is the same as that given in \cite{BH_Period}.  We fix any positive real number $B$ and any compact set $K \subset \Omega$.  Notice that $T^\epsilon$ converges in $L^2$ to $T$ on $[-A,A]\times K$.  Hence, to show that
\begin{equation}\label{L2_periodicity}
	\int_{[-A,A]\times K} \left[T_\delta(t + \frac{\ell}{c_\delta}, x, z) - T_\delta(t, x - \ell, z) \right]^2dxdzdt = 0
\end{equation}
it suffices to prove that
$$\lim_{\epsilon\to 0} \int_{[-A,A]\times K} \left[T_\delta^\epsilon(t + \frac{\ell}{c_\delta}, x, z) - T_\delta^\epsilon(t, x - \ell, z) \right]^2dxdzdt = 0.$$
To this end, we first notice that $T_\delta^\epsilon(t + \frac{\ell}{c_{\epsilon,\delta}}, x, z) = T_\delta^\epsilon(t, x - \ell, z)$.  Hence we calculate:

\begin{equation*}
	\begin{split}
		\int_{[-A,A]\times K} &\left[T_\delta^\epsilon(t + \frac{\ell}{c_\delta}, x, z) - T_\delta^\epsilon(t, x - \ell, z) \right]^2dxdzdt\\
			&= \int_{[-A,A]\times K} \left[T_\delta^\epsilon(t + \frac{\ell}{c_\delta}, x, z) - T_\delta^\epsilon(t + \frac{\ell}{c_{\epsilon,\delta}}, x , z) \right]^2dxdzdt\\
			&\leq \left|\frac{\ell}{c_{\epsilon,\delta}} - \frac{\ell}{c_\delta}\right|^2\int_{\R\times K} |(T_\delta^\epsilon)_t|^2dxdzdt\\
			&\leq C\cdot \left|\frac{\ell}{c_{\epsilon,\delta}} - \frac{\ell}{c_\delta}\right|^2.
	\end{split}
\end{equation*}

Since $c_\epsilon \to c$, then we have that equation \eqref{L2_periodicity} holds.  Hence $T_\delta(t + \ell/c, x, z) = T_\delta(t, x - \ell, z)$ holds almost everywhere.  Of course, by parabolic regularity, $T_\delta$ is continuous.  Hence, the equality holds everywhere.  The same argument works for $u_\delta$ as well.
\end{proof}

Finally, in order to take the limit as $\delta$ tends to zero, we need to have some uniform bounds on our functions.  We will do that here.  First, notice that the $L^2$ gradient bounds we obtained in Section \ref{solutions_on_the_infinite_cylinder} did not depend on $\delta$, except for the bound on the time derivative of $T_\delta$.  These bounds, along with parabolic regularity, are enough to get upper bounds on the $X^{1,2}$ norm of $\omega$.  This gives us an upper bound on the $C^{0,\alpha}$ norm of $u_\delta$ as in Lemma~\ref{flow_bounds}.  Arguing as in Corollary \ref{front_speed_upper_bound_independent_of_epsilon} will then provide an upper bound on the front speed, $c_\delta$, which, in turn, provides an upper bound on the $L^2$ norm of the time derivative of $T_\delta$.  From here, parabolic and elliptic regularity gives us Sobolev and H\"older bounds independent of $\delta$ that we summarize in the lemma below.

\begin{lem}\label{parabolic_bounds}
There exists a constant $C>0$, independent of $\delta$ such that
\[
	|c_\delta| + \|T_\delta\|_{C^{1+\alpha, 2+\alpha}}+\|\omega_\delta\|_{C^{1+\alpha, 2+\alpha}}+\|\Psi_\delta\|_{C^{1+\alpha, 2+\alpha}}+\|u_\delta\|_{C^{1+\alpha, 2+\alpha}} \leq C.
\]
Moreover, for every compact set $K \subset \Omega$, there exists a constant $C_K = C(\diam(K))>0$, which depends only on the diameter of $K$ such that
\[
	\|T_\delta\|_{X^{1,2}(\R\times K)}+\|\omega_\delta\|_{X^{1,2}(\R\times K)}+\|\Psi_\delta\|_{X^{1,2}(\R\times K)}+\|u_\delta\|_{X^{1,2}(\R\times K)} \leq C_K.
\]
\end{lem}

In order to finish Theorem \ref{the_theorem}, we need to take the limit $\delta \to 0$.  Afterwards, we need to check that the front speed, $c$, is positive and that our solutions are non-trivial, satisfy condition \eqref{normalization}, and satisfy equation \eqref{pulsating_front}.  Most of these claims are proved similarly as the analogous results from this section.

\begin{proof}[Proof of Theorem \ref{the_theorem}]
Notice that Lemma \ref{parabolic_bounds} gives us uniform bounds in Sobolev and H\"older spaces which are independent of $\delta$.  As a result, we can take limits in these spaces to obtain functions $T$, $u$, $\Psi$, and $\omega$.  By the theory of convolutions, we get that $\tilde{\omega}_\delta$ and $\omega_\delta$ converge to the same function, $\omega$.  Hence, these functions satisfy the stream function formulation of our problem, which in turn implies that they satisfy the system \eqref{the_PDE} - \eqref{boundary_conditions}.

Arguing exactly as in Proposition \ref{pos_front_speed_unregularized}, we see that $c_\delta$ is bounded away from zero.  Moreover, arguing as in Lemma \ref{front_speed_upper_bound_independent_of_epsilon} shows that it is bounded uniformly above.  Hence $c_\delta$ converges to a positive number $c$.

Since we have convergence in the H\"older norms as $\delta$ tends to zero, and since, for each $\delta$, these functions satisfy the condition \eqref{pulsating_front} and normalization condition \ref{normalization}, then our limiting functions will satisfy these as well.

Hence we need only check their limits as $x\to\pm \infty$.  To do this it is convenient to look in the moving frame, where the equivalent limit to check is what happens when $s \to \pm \infty$.  The limits are easiest to prove for $\Psi$, $u$, and $\omega$.  The bounds in Lemma \ref{parabolic_bounds} imply that these functions are globally bounded in $L^2$ and $C^{1,\alpha}$ in the moving frame.  Hence as $s\to\pm \infty$, these functions tend to zero.

Now we will check the limits of $T$.  First, arguing exactly as in Lemmas \ref{right_uniform_limit_zero_moving}, we can find a universal constant $R>0$, which does not depend on $\delta$ such that
\[
	c/10 = \max_{\substack{(x,z)\in \Omega_p,\\ s \geq R}} u^{m}_\delta(s, x, z).
\]
Hence arguing as in Corollary \ref{reg_right_limit}, we can bound $T^m$ by an exponential.  It follows that $T^m(s,x,z) \to 0$ as $s\to \infty$.  This is equivalent to $T(t,x,z) \to 0$ as $x \to \infty$.

To get the limit as $s$ tends to $-\infty$, one simply argues as in Lemma \ref{bdry_deriv_dies}.  To see that $\theta_-$ is positive, we notice that integrating the equation for $T^m$ implies that
\[
	c |\Omega_p|\theta_- = \int_{\R\times\Omega_p} f(T^m) dxdzds.
\]
Hence if $\theta_- = 0$, then $T \leq \theta_0$ on $\R\times \Omega$.  In this case, the parabolic maximum principle applied in the stationary frame, along with Lemma \ref{pinned_down_at_middle} imply that $T \equiv \theta_0$.  This cannot be true since $T(t,x,z) \to 0$ as $t\to -\infty$.  Hence $\theta_- > 0$.  Notice that this also implies that
\[
	\int_{\R\times \Omega_p} f(T^m) dxdzds >0, 
\]
and, hence, $T$ takes values larger than $\theta_0$.

Finally, we show that if \eqref{ignition_smallness_condition} is satisfied, then $\theta_- = 1$.  This proof is similar, though not identical, to the proofs in~\cite{BCR, Lew}.  First, we assume that $p \in (2,3]$ since if $p > 3$ then $(T - \theta_0)^p \leq (T - \theta_0)_+^3$.  Hence, in this case, we may assume condition \eqref{ignition_smallness_condition} holds with $p = 3$.  To begin notice that since
\[
	\int_{\R\times \Omega_p} f(T^m) dxdzds  < \infty,
\]
then $\theta_- \in (0,\theta_0]\cup \{1\}$.  Hence we suppose that $\theta_-\leq \theta_0$.  Now let $K = \{(x,z) \in \Omega: x \in [0,\ell]\}$.  Define
\[
	M(t) = \max_{(x,z) \in K} T(t,x,z) ~~\text{ and }~~ m(t) = \min_{(x,z)\in K} T(t,x,z),
\]
for any $t\in \R$.  First notice that for any $t\in \R$, we have
\begin{equation}\label{max_min}
\begin{split}
	M(t) - m(t)
		&= M(t) - \frac{1}{|K|} \int_K T(t,x,z)dxdz + \frac{1}{|K|}\int_K T(t,x,z)dxdz - m(t)\\
		&\leq 2 \left\| T(t,\cdot,\cdot) - |K|^{-1}\int_K T(t,x,z)dxdz\right\|_{L^\infty(K)}\\
		&\leq C\left\| T(t,\cdot,\cdot) - |K|^{-1}\int_K T(t,x,z)dxdz\right\|_{W^{1,p}(K)}\\
		&\leq C \|\nabla T\|_{L^p(K)}.
\end{split}
\end{equation}
The second inequality comes from the Sobolev inequality and the third inequality comes from the Poincar\'e inequality.  We need a bound on the $L^p$ norm of $\nabla T$, which we will obtain by getting a bound on the $L^3$ norm of $\nabla T$ and interpolating with the $L^2$ bound of $\nabla T$.  To this end we follow the development in~\cite{Lew}.  First notice that the work in Lemma \ref{temperature_time_derivative} shows that if $\theta_- \leq \theta_0$ then we have that
\[
	\int_{\R\times K} |T_t|^2dxdzdt \leq C \int_{\R\times K} |\nabla T|^2 dxdzdt.
\]
Multiplying \eqref{the_PDE} by $T|\nabla T|$ and integrating over $\R\times K$ gives us
\begin{equation*}
	\begin{split}
		\left| \int T |\nabla T| \Delta T dtdxdz\right|
			&= \left| \int T |\nabla T| T_t + \int (u\cdot \nabla T)T|\nabla T| - \int f(T)T|\nabla T|  \right|\\
			&\leq (C + \|u\|_{L^\infty}) \|\nabla T\|_{L^2}^2 + \int f(T) |\nabla T|\\
			&\leq C \|\nabla T\|_{L^2}^2 + C \int f(T).
	\end{split}
\end{equation*}
To understand the right hand side, we notice that
\begin{equation*}
	\int |\nabla T|^3  = -\int T|\nabla T| \Delta T - \int T (\nabla T \cdot \nabla|\nabla T|).
\end{equation*}
We need to understand the last term on the right.  To this end, the argument in Lemma 3.4 of~\cite{Lew} implies that
\[
\|\nabla^2 T\|_{L^2} \leq C( \| \Delta T\|_{L^2} + \|\nabla T\|_{L^2}).
\]
Also, \eqref{the_PDE}  implies that
\[
	\|\Delta T\|_{L^2}
		\leq \| T_t \|_{L^2} + \| u\cdot \nabla T\|_{L^2} + \|f(T)\|_{L^2}.
\]
Combining all this gives us
\begin{equation}\label{L3_bound}
	\int |\nabla T|^3 \leq C \left( \int |\nabla T|^2 + \int f(T)\right).
\end{equation}
We also observe that
\[
	\theta_- = \int_{\R\times K} f(T) \frac{dtdxdz}{|K|}, ~~ \frac{\theta_-^2}{2} + \int_{\R\times K} |\nabla T(t,x,z)|^2 \frac{dtdxdz}{|K|} = \int_{\R\times K} Tf(T) \frac{dtdxdz}{|K|}
\]
so that
\begin{equation}\label{temperature_gradient_ignition_relationship}
	\int_{\R\times K} |\nabla T|^2 dtdxdz = \int_{\R\times K} \left(T - \frac{\theta_-}{2}\right)f(T) dtdxdz \geq  \frac{\theta_0}{2} \int f(T)
\end{equation}
Hence equations \eqref{max_min}, \eqref{L3_bound}, and \eqref{temperature_gradient_ignition_relationship} give us that
\begin{equation}\label{max_min_gradient_temperature}
	\int_{\R} (M(t) - m(t))^p dt
		\leq C\int_{\R\times K} |\nabla T|^2 dtdxdz.
\end{equation}
Define $C_{\Omega,p}$ in \eqref{ignition_smallness_condition} to be the reciprocal of the constant in the above equation, namely
\[
C_{\Omega,p} = \frac{1}{C}\]
with $C$ as in \eqref{max_min_gradient_temperature}.

We claim that $m(t)$ is non-decreasing in $t$.  To see this, one can first show it on the finite domain using the maximum principe since $T_{n, \delta}^{\epsilon}$ satisfies an elliptic equation.  Then this property will hold through the limits since when we take $n\to\infty$ and $\delta\to 0$ we have convergence in H\"older norms and when we take $\epsilon \to 0$ we have bounds on the $L^2$ and $L^\infty$ norms of $T$ as in Lemma \ref{pinned_down_at_middle}.  Because $m$ is non-decreasing and $T$ is non-constant, then it follows that $m(t) < \theta_0$ for all $t$.

Since $m(t) < \theta_0$, $M(t) \geq T(t,x,z)$ and by \eqref{max_min_gradient_temperature}, we have that
\[
	\int_{\R\times K} |\nabla T(t,x,z)|^2
		\geq C_{\Omega,p}\int_{\R} (M(t) - m(t))^pdt
		\geq C_{\Omega,p}\int_{\R\times K} (T(t,x,z) - \theta_0)^p_+ dtdxdz.
\]
Hence using \eqref{ignition_smallness_condition} we obtain
\[
	\begin{split}
	C_{\Omega,p}\int_{\R\times K} \left(T - \frac{\theta_-}{2}\right)(T- \theta_0)^p_+ dtdxdz
		&\geq \int_{\R\times K} \left(T - \frac{\theta_-}{2}\right) f(T)dtdxdz\\
		&= \int_{\R\times K} |\nabla T|^2 dtdxdz\\
		&\geq C_{\Omega,p}\int_{\R\times K} (T - \theta_0)^p_+ dtdxdz.
	\end{split}
\]
The left hand side is smaller than the right hand side unless $T \equiv \theta_0$.  This can't occur because of Corollary \ref{right_uniform_limit_zero}.  Hence, we have reached a contradiction, implying that $\theta_- > \theta_0$.  We then conclude that $\theta_- = 1$.
\end{proof}

\bibliography{henderson-pulsating_fronts.bib}{}
\bibliographystyle{amsplain}

\end{document}